\documentclass[a4paper,11pt,reqno]{amsart}

\usepackage[left=1in,right=1in,top=1in,bottom=1in]{geometry}
\usepackage{enumitem}
\usepackage{amssymb,amsmath,latexsym,amsfonts,amsbsy,amsthm,mathtools,color}
\usepackage{float}
\usepackage[colorlinks=true,citecolor=blue,filecolor=blue,linkcolor=blue,urlcolor=blue]{hyperref}

\usepackage[T1]{fontenc}
\usepackage{lmodern}
\usepackage[expansion=false]{microtype}
\usepackage{graphicx}
\usepackage{upgreek}
\usepackage{mathrsfs}
\usepackage{needspace}

\newcommand{\be}{\begin{equation}}
\newcommand{\ee}{\end{equation}}

\newcommand{\cE}{\mathcal{E}}

\newcommand{\cL}{\mathcal{L}}

\newcommand{\R}{\mathbb{R}}

\newcommand{\abs}[1]{\left\lvert#1\right\rvert}
\newcommand{\norm}[1]{\left\lVert#1\right\rVert}

\newcommand{\dd}{\,\mathrm{d}}

\DeclareMathOperator{\dist}{dist}

\DeclareMathOperator{\supp}{supp}

\DeclareMathOperator{\tr}{tr}

\numberwithin{equation}{section}

\theoremstyle{plain}
\newtheorem{thm}{Theorem}[section]
\newtheorem{cor}[thm]{Corollary}

\newtheorem{lem}[thm]{Lemma}
\newtheorem{prop}[thm]{Proposition}
\newtheorem{assum}[thm]{Assumption}

\theoremstyle{definition}
\newtheorem{defn}[thm]{Definition}

\theoremstyle{remark}
\newtheorem{rem}[thm]{Remark}

\newcommand{\eps}{\varepsilon}
\newcommand{\cH}{\mathcal H}
\newcommand{\cM}{\mathcal M}
\newcommand{\Sstar}{\mathfrak S_\star}

\newcommand{\nustar}{\upnu_\star}
\newcommand{\zetastar}{\upzeta_\star}

\newcommand{\Vbold}{\mathbf V}
\newcommand{\nOmega}{\vec n_\Omega}
\newcommand{\diver}{\operatorname{div}}
\newcommand{\Per}{\operatorname{Per}}

\title[Boundary vectorial Allen--Cahn asymptotics]
{Boundary asymptotics for two-dimensional vectorial Allen--Cahn systems}
\author{Zhiyuan Dai}
\address{School of Mathematical Sciences, Zhejiang University, Hangzhou 310058, China}
\email{12635038@zju.edu.cn}

\author{Haotong Fu}
\address{School of Mathematical Sciences, Peking University, Beijing 100871, China}
\email{2301110012@pku.edu.cn}

\author{Huaijie Wang}
\address{School of Mathematical Sciences, Peking University, Beijing 100871, China}
\email{huaijie\_wang@163.com}

\author{Wei Wang}
\address{School of Mathematical Sciences, Peking University, Beijing 100871, China}
\email{wwmath166@outlook.com}
\email{2201110024@stu.pku.edu.cn}
\date{August 30, 2026}

\begin{document}

\begin{abstract}
We study asymptotic properties of critical points of the two-dimensional vectorial Allen–Cahn energy with finitely many non-degenerate wells, subject to a homogeneous Neumann boundary condition. For any sequence with uniformly bounded energy, we prove that the limiting full and potential measures are supported on a closed
countably $1$-rectifiable set up to the boundary and satisfy the discrepancy
relations; the potential measure defines a free-boundary stationary
rectifiable varifold. Boundary mass may occur: weights are constant on regular boundary arcs, finite-type junctions obey projected balance, and the sole non-boundary branch meets the boundary orthogonally. This provides a Neumann-boundary extension of Bethuel’s planar interior theory. The key idea is to read the boundary geometry from the limiting stress–energy tensor, which identifies the potential-energy measure as the stationary interfacial measure even in the presence of boundary concentration.
\end{abstract}


\maketitle

\section{Introduction}
\label{sec:introduction}

Let $\Omega\subset\R^2$ be a bounded domain with $C^3$ boundary. For each small parameter $\eps\in (0,1]$, we consider the Allen-Cahn energy functional
\begin{equation}\label{eq:intro-energy}
 \cE_\eps(u;\Omega)
 :=\int_\Omega e_\eps(u)\,\dd x,
 \quad
 e_\eps(u)
 :=\frac{\eps}{2}\abs{\nabla u}^2+\frac{1}{\eps} V(u),
\end{equation}
where $u:\Omega\to\mathbb R^k$, $V$ is a smooth multi-well potential with non-degenerate wells, satisfying
\begin{equation}\label{eq:intro-vacuum}
 V\in C^\infty(\R^k;[0,+\infty)),
 \quad
 \Sigma:=V^{-1}(0)=\{\upsigma_1,\ldots,\upsigma_q\},
 \quad q\geq2.
\end{equation}
The precise hypotheses on $V$ are collected
in Assumption~\ref{assum:potential} below. The scalar double-well case, corresponding to $k=1$ and $q=2$, is the
classical prototype and will serve as the main point of comparison.

The energies considered here originate in the diffuse-interface models of
Cahn--Hilliard and Allen--Cahn
\cite{CahnHilliard1958,AllenCahn1979}.  For minimizing sequences, the scalar
sharp-interface limit is described by the Modica--Mortola theorem
\cite{ModicaMortola1977,Modica1987}; the multi-well vectorial extensions of
Baldo and Fonseca--Tartar identify the corresponding weighted partition
problem \cite{Baldo1990,FonsecaTartar1989}.  The theory of arbitrary critical points was likewise first developed in the scalar double-well setting. Modica's gradient estimate controls the
discrepancy between gradient and potential energy \cite{Modica1985}.
Hutchinson and Tonegawa developed the bounded-energy argument leading to an
integral stationary varifold \cite{HutchinsonTonegawa2000}.  For the elliptic
Neumann problem, Tonegawa obtained boundary monotonicity under a convexity
assumption on the domain \cite{Tonegawa2003}.

For critical points, the vectorial problem is substantially different: there is no ordering of
the target, no scalar Modica inequality, and no general equipartition
identity.  Bethuel's two-dimensional theorem \cite{Bethuel2025} replaces
these scalar tools by a novel $\eps$-level PDE analysis and identifies a
countably $1$-rectifiable concentration set.  In the parabolic, well-prepared
two-phase regime, Moser proved convergence for scalar and vector-valued
Neumann problems to mean-curvature flow with a right-angle contact condition
\cite{Moser2023}.  The present work instead treats arbitrary bounded-energy
elliptic critical points.

Boundary convergence is already understood more fully in the scalar theory.
Mizuno and Tonegawa proved convergence in the parabolic setting to a Brakke
flow satisfying a generalized right-angle boundary condition
\cite{MizunoTonegawa2015}, extending Ilmanen's measure-theoretic approach
\cite{Ilmanen1993}; Kagaya later removed the convexity restriction from
related boundary arguments \cite{Kagaya2019}.  For elliptic critical points,
Li, Parise, and Sarnataro proved convergence up to the boundary to an
integer-rectifiable varifold stationary with free boundary
\cite{LiPariseSarnataro2024}.  Boundary-supported mass may nevertheless
occur, as shown by the scalar solutions in
\cite{MalchiodiNiWei2007,MalchiodiWei2007}.  This makes free-boundary
varifolds \cite{GruterJost1986,DeMasi2021,Simon1983} the appropriate
geometric framework.

\subsection{Setting and main results}
\label{subsec:intro-problem}
 Let $\nOmega$ denote the outer unit normal to $\partial\Omega$. We work directly with finite-energy weak solutions of the Euler-Lagrange equation
\begin{equation}\label{eq:intro-system}
 \begin{cases}
 -\eps\Delta u_\eps+\eps^{-1}\nabla V(u_\eps)=0
     &\text{in }\Omega,\\
 \partial_{\nOmega}u_\eps=0
     &\text{on }\partial\Omega.
 \end{cases}
\end{equation}
More precisely, a finite-energy weak solution is a map
$u_\eps\in H^1(\Omega;\R^k)$ such that
$\nabla V(u_\eps)\in L^1(\Omega;\R^k)$ and, for every
$\varphi\in H^1(\Omega;\R^k)\cap L^\infty(\Omega;\R^k)$, it satisfies
\begin{equation}\label{eq:intro-weak-system}
 \int_\Omega
 \left[
  \eps\,\partial_i u_\eps^\alpha\partial_i\varphi^\alpha
  +\eps^{-1}\partial_\alpha V(u_\eps)\varphi^\alpha
 \right]\dd x=0.
\end{equation}

The fundamental tensor in the paper is the \emph{stress--energy tensor}
\begin{equation}\label{eq:intro-stress-energy-tensor}
 T_\eps(u)
 :=e_\eps(u)I_2-\eps(\nabla u)^T\nabla u,
 \quad
 (T_\eps(u))_{ij}
 =e_\eps(u)\delta_{ij}
  -\eps\,\partial_i u\cdot\partial_j u.
\end{equation}
A direct calculation yields
\begin{equation}\label{eq:intro-stress-divergence}
 \partial_j(T_\eps(u))_{ij}
 =\left[-\eps\Delta u+\eps^{-1}\nabla V(u)\right]
   \cdot\partial_i u.
\end{equation}
Hence every solution of \eqref{eq:intro-system} satisfies
\begin{equation}\label{eq:intro-strong-stationary-identity}
 \diver T_\eps(u_\eps)=0\quad\text{in }\Omega.
\end{equation}
We call \eqref{eq:intro-strong-stationary-identity}, and its weak form
below, the \emph{stationary identity}.  If
$X\in C^1(\overline\Omega;\R^2)$, then
\begin{equation}\label{eq:intro-DX}
 DX:=(\partial_iX^j)_{i,j=1}^2.
\end{equation}
The Neumann condition implies
$T_\eps(u_\eps)\nOmega=e_\eps(u_\eps)\nOmega$ on $\partial\Omega$.
Consequently, whenever $X\cdot\nOmega=0$ on $\partial\Omega$,
\begin{equation}\label{eq:intro-stationary-identity}
 \int_\Omega T_\eps(u_\eps):DX\,\dd x=0.
\end{equation}
Identity \eqref{eq:intro-stationary-identity} is the analytic source of the
free-boundary stationarity obtained in the limit.

We study sequences $\eps_j\downarrow0$ satisfying the natural interfacial
bound
\begin{equation}\label{eq:intro-energy-bound-overview}
 \sup_j\cE_{\eps_j}(u_{\eps_j};\Omega)<+\infty.
\end{equation}
Theorem~\ref{thm:main} states that the full and potential energy measures
concentrate on a closed countably $1$-rectifiable set
$\Sstar\subset\overline\Omega$.  Their limiting densities obey Bethuel's
tensorial discrepancy relations up to the boundary, and the potential-energy
measure determines a rectifiable varifold stationary with free boundary.
Boundary-supported mass is allowed; thus free-boundary stationarity, rather
than a blanket pointwise orthogonality assertion, is the natural conclusion.

We use the following assumptions throughout the paper.

\begin{assum}\label{assum:potential}
The potential $V\in C^\infty(\R^k;[0,\infty))$ satisfies:
\begin{enumerate}[label=$(\theenumi)$]
\item the vacuum set is finite,
\[
 \Sigma:=V^{-1}(0)=\{\upsigma_1,\ldots,\upsigma_q\},
 \quad q\geq2;
\]
\item $D^2V(\upsigma_i)$ is positive definite for every $i$;
\item there are $R_\infty,\alpha_\infty>0$ such that
\[
 y\cdot\nabla V(y)\geq\alpha_\infty\abs{y}^2
 \quad\text{for }\abs{y}\geq R_\infty,
 \quad
 V(y)\to\infty\quad\text{as }\abs{y}\to\infty.
\]
\end{enumerate}
\end{assum}

For a solution $u_\eps$, define the Radon measures
\begin{equation}\label{eq:intro-measures}
 \begin{aligned}
 \upnu_\eps
 &:=\left[\frac{\eps}{2}\abs{\nabla u_\eps}^2
       +\frac{1}{\eps}V(u_\eps)\right]\dd x,\\
 \upzeta_\eps&:=\frac{1}{\eps}V(u_\eps)\dd x,\\
 \upmu_{\eps,ij}&:=\eps\,\partial_i u_\eps\cdot
       \partial_j u_\eps\dd x.
\end{aligned}
\end{equation}
Set
\begin{equation}\label{eq:intro-stress-measure}
 \upmu_\eps:=(\upmu_{\eps,ij})_{i,j=1}^2,
 \quad
 T_\eps:=\upnu_\eps I_2-\upmu_\eps.
\end{equation}
Thus $T_\eps$ is the matrix-valued Radon measure
$T_\eps(u_\eps)\,\dd x$ associated with the stress--energy tensor
\eqref{eq:intro-stress-energy-tensor}; we use the same letter for the tensor
field and its associated measure when no confusion is possible.
Whenever
$\upnu_{\eps_j}\stackrel{*}{\rightharpoonup}\nustar$ and
$\upmu_{\eps_j,ij}\stackrel{*}{\rightharpoonup}\upmu_{\star,ij}$, write
\begin{equation}\label{eq:intro-limit-stress}
 \upmu_\star:=(\upmu_{\star,ij})_{i,j=1}^2,
 \quad
 T:=\nustar I_2-\upmu_\star.
\end{equation}
For a finite Radon measure $\lambda$ on $\overline\Omega$,
\begin{equation}\label{eq:intro-measure-notation}
 (\lambda\llcorner A)(B):=\lambda(A\cap B),
 \quad
 \supp\lambda
 :=\overline\Omega\backslash
 \bigcup\{G\subset\overline\Omega:G\text{ open},\ \lambda(G)=0\}.
\end{equation}
Weak-star convergence $\lambda_j\stackrel{*}{\rightharpoonup}\lambda$ means
\begin{equation}\label{eq:intro-weak-star-definition}
 \int_{\overline\Omega}\phi\,\dd\lambda_j
 \to
 \int_{\overline\Omega}\phi\,\dd\lambda
 \quad\forall\phi\in C(\overline\Omega).
\end{equation}

A Borel set $\mathcal S\subset\R^2$ is \emph{countably
$1$-rectifiable} if there are Lipschitz maps $f_m:\R\to\R^2$ such that
\begin{equation}\label{eq:intro-rectifiable-definition}
 \cH^1\left(\mathcal S\backslash\bigcup_{m=1}^\infty f_m(\R)\right)=0,
\end{equation}
where $\cH^1$ denotes one-dimensional Hausdorff measure.  Its approximate
tangent line, defined at $\cH^1$-almost every $x\in\mathcal S$, is denoted by
$T_x\mathcal S$.  If $E\subset\Omega$ is measurable, its relative perimeter
is
\begin{equation}\label{eq:intro-perimeter-definition}
 \Per(E;\Omega):=\abs{D\mathbf1_E}(\Omega).
\end{equation}
A family $(E_\alpha)_{\alpha=1}^q$ is a finite-perimeter partition when the
sets are pairwise disjoint up to $\cL^2$-null sets, cover $\Omega$ up to a
$\cL^2$-null set, and
$\sum_{\alpha=1}^q\Per(E_\alpha;\Omega)<+\infty$.

For a positive function
$\Uptheta\in L^1_{\mathrm{loc}}(\cH^1\llcorner\mathcal S)$ on a countably
rectifiable $\mathcal S$, the corresponding rectifiable $1$-varifold is denoted by
$\Vbold(\mathcal S,\Uptheta)$ and has weight
\begin{equation}\label{eq:intro-varifold-definition}
 \norm{\Vbold(\mathcal S,\Uptheta)}
 :=\Uptheta\,\cH^1\llcorner\mathcal S.
\end{equation}
Its first variation is
\begin{equation}\label{eq:intro-first-variation-definition}
 \delta\Vbold(\mathcal S,\Uptheta)(X)
 :=\int_{\mathcal S}\diver_{T_x\mathcal S}X\,
       \Uptheta\,\dd\cH^1,
 \quad
 \diver_{T_x\mathcal S}X:=(\tau\otimes\tau):DX,
\end{equation}
where $\tau$ is either unit vector spanning $T_x\mathcal S$; the expression
is independent of its sign.  The varifold is \emph{stationary with free
boundary} if this first variation vanishes for every
$X\in C^1(\overline\Omega;\R^2)$ with
$X\cdot\nOmega=0$ on $\partial\Omega$.
For a line $L=\operatorname{span}\{\tau\}$ we write
$P_L:=\tau\otimes\tau$ for the orthogonal projection onto $L$; in
particular,
\begin{equation}\label{eq:intro-boundary-projection}
 P_{T_p\partial\Omega}
 :=\tau_{\partial\Omega}(p)\otimes\tau_{\partial\Omega}(p).
\end{equation}

\begin{defn}\label{defn:finite-type-boundary}
Let $\mathcal S\subset\overline\Omega$ be countably $1$-rectifiable and let
$\lambda=\vartheta\,\cH^1\llcorner\mathcal S$.  A point
$p\in\mathcal S\cap\partial\Omega$ is a \emph{finite-type boundary junction
of $(\mathcal S,\lambda)$} if there are a neighborhood $U$ of $p$ and
finitely many regular embedded $C^1$ half-arcs
$\gamma_\ell:[0,r_\ell)\to\overline\Omega$ such that
$\gamma_\ell(0)=p$, the open arcs are pairwise disjoint, and
\[
 \mathcal S\cap U=\bigcup_\ell\gamma_\ell([0,r_\ell));
\]
thus no residual branch or weight is present in $U$.
The outgoing tangent is
$\tau_\ell:=\frac{\gamma_\ell'(0)}{\abs{\gamma_\ell'(0)}}$.  A
\emph{boundary-supported regular arc} is an open $C^1$ subarc of
$\mathcal S\cap\partial\Omega$ that has a neighborhood containing no other branch
of $\mathcal S$.
\end{defn}

\Needspace{14\baselineskip}
\par\addvspace{\medskipamount}
\begin{thm}[Neumann-boundary compactness and stationarity]
\label{thm:main}
Let $\Omega\subset\R^2$ be a bounded $C^3$ domain, let $V$ satisfy
Assumption~\ref{assum:potential}, and let
$\eps_j\in(0,1]$ satisfy $\eps_j\downarrow0$.  Let
$u_j:=u_{\eps_j}\in H^1(\Omega;\R^k)$ be weak solutions of
\eqref{eq:intro-system}.  Suppose that
\begin{equation}\label{eq:intro-energy-bound}
\cE_{\eps_j}(u_j;\Omega)\leq M_0
\quad\text{for some }M_0<+\infty.
\end{equation}
Then, after passing to a subsequence, there are a closed countably
$1$-rectifiable set $\Sstar\subset\overline\Omega$, a finite-perimeter
partition $(E_\alpha)_{\alpha=1}^q$ of $\Omega$, Radon measures
$\nustar,\zetastar$, and finite signed Radon measures
$\upmu_{\star,ij}$ on $\overline\Omega$ with the
following properties.

\begin{enumerate}[label=$(\theenumi)$]
\item The phase fields satisfy
\[
 u_j\to u_0:=\sum_{\alpha=1}^q\upsigma_\alpha\mathbf1_{E_\alpha}
 \quad\text{in }L^1(\Omega),
\]
with
$\sum_{\alpha=1}^q\Per(E_\alpha;\Omega)\leq C(V)M_0$, and the convergence is uniform on
compact subsets of $\overline\Omega\backslash \Sstar$ to a locally constant
map $\bar u_0:\overline\Omega\backslash \Sstar\to\Sigma$ that agrees almost
everywhere in $\Omega\backslash \Sstar$ with
$\sum_{\alpha=1}^q\upsigma_\alpha\mathbf1_{E_\alpha}$.

\item The concentration set satisfies
\[
 \Sstar=\supp\nustar=\supp\zetastar,
 \quad
 \cH^1(\Sstar)\leq C(V,\Omega)M_0.
\]

\item The measures in \eqref{eq:intro-measures} converge weakly star to
$\nustar$, $\zetastar$, and $\upmu_{\star,ij}$, respectively, and
\begin{equation}\label{eq:intro-absolute-continuity}
 \nustar=e\,\cH^1\llcorner\Sstar,
 \quad
 \zetastar=\Uptheta\,\cH^1\llcorner\Sstar,
 \quad
 \upmu_{\star,ij}=m_{ij}\,\cH^1\llcorner\Sstar.
\end{equation}
There are constants $c_0=c_0(V,\Omega)>0$ and
$C_0,C_1<\infty$, depending only on $V$, $\Omega$, and $M_0$, such that
\begin{equation}\label{eq:intro-density-comparison}
 c_0\leq\Uptheta\leq e\leq C_0\Uptheta\leq C_1,
 \quad
 \abs{m_{ij}}\leq2e
 \quad\text{for }\cH^1\text{-almost every }x\in \Sstar.
\end{equation}

\item At $\cH^1$-almost every point of $\Sstar$, let $(\tau,n)$ denote the
approximate tangent-normal frame of $\Sstar$.  For
$a,b\in\{\tau,n\}$ set $m_{ab}:=a^Tm b$.  Then
\begin{equation}\label{eq:intro-discrepancy}
 2\Uptheta=m_{nn}-m_{\tau\tau},
 \quad
 m_{n\tau}=0,
 \quad
 e=m_{nn}.
\end{equation}

\item The rectifiable varifold $\Vbold(\Sstar,\Uptheta)$ is stationary with
free boundary:
\begin{equation}\label{eq:intro-free-stationarity}
 \int_{\Sstar}\diver_{T_x\Sstar}X\,\dd\zetastar(x)=0
\end{equation}
for every $X\in C^1(\overline\Omega;\R^2)$ satisfying
$X\cdot\nOmega=0$ on $\partial\Omega$.

\item In the interior, the local straight-segment and constant-density
description of \cite[Theorem 1.3]{Bethuel2025} holds away from an
$\cH^1$-null set.  On a boundary-supported regular arc with no incident
interior branch, $\Uptheta$ has a positive constant representative.  At a
finite-type boundary junction $p$ of $(\Sstar,\zetastar)$ in the sense of
Definition~\ref{defn:finite-type-boundary}, each incident punctured branch
carries a positive constant weight $\Uptheta_\ell$, and
\begin{equation}\label{eq:intro-boundary-balance}
 P_{T_p\partial\Omega}
 \left(\sum_\ell\Uptheta_\ell\tau_\ell\right)=0.
\end{equation}
In particular, if the decomposition in
Definition~\ref{defn:finite-type-boundary} consists of exactly one
non-boundary half-arc and no boundary-supported or additional branch, then
that half-arc meets $\partial\Omega$ orthogonally.
\end{enumerate}
\end{thm}

\Needspace{7\baselineskip}
\begin{rem}
The conclusions have three consequences that delimit the theorem's scope.
\begin{enumerate}[label=$(\theenumi)$]
\item Theorem~\ref{thm:main} extends the compactness, rectifiability,
stationarity, absolute-continuity, and discrepancy conclusions of
\cite[Theorems 1.2, 1.5, 1.8, and 1.9]{Bethuel2025} to the closure of a
smooth planar domain.  It retains the local-segment conclusion of
\cite[Theorem 1.3]{Bethuel2025} in the interior and replaces its literal
boundary analogue, which is false in the presence of boundary-supported
mass, by constancy on regular boundary arcs and the projected finite-junction
law.  The scalar theorem proves both equipartition and integer multiplicity
\cite[Theorem 3.3]{LiPariseSarnataro2024}; neither is asserted here.

\item Boundary-supported mass cannot be excluded under the hypotheses of
Theorem~\ref{thm:main}.  Indeed, take the quartic potential
$W(s)=\frac{(1-s^2)^2}{4}$ and, for a scalar solution $v_\eps$, set
\begin{equation}\label{eq:scalar-embedding}
 V(y_1,\ldots,y_k)
 :=W(y_1)+\frac{1}{2}\sum_{a=2}^k y_a^2,
 \quad
 u_\eps=(v_\eps,0,\ldots,0).
\end{equation}
Each scalar Neumann solution is therefore a vectorial one.  The
boundary-clustered solutions constructed in
\cite[Theorem 1.1]{MalchiodiNiWei2007} and
\cite[Theorem 1.1]{MalchiodiWei2007} rule out blanket orthogonality.

\item The stationarity conclusion is attached to the potential measure
$\zetastar$, not to the full energy measure $\nustar$.  This distinction is
intrinsic to the vectorial discrepancy relation
\eqref{eq:intro-discrepancy}.
\end{enumerate}
\end{rem}

\subsection{Difficulties and strategies}
\label{subsec:intro-strategy}

The interior compactness and discrepancy theory is supplied by
Bethuel's theorem.  The proof therefore concentrates on the passage from
interior balls to balls centered on $\partial\Omega$.  Two principal
difficulties determine the strategy.
First, conformal flattening followed by even reflection does not preserve the
constant-coefficient equation.  More precisely, for a conformal boundary
chart $\Psi:B_{2\rho}^+\to\Omega$ and
\begin{equation}\label{eq:architecture-reflection}
 v_j:=u_j\circ\Psi,
 \quad
 a(x):=\abs{\Psi'(x)}^2,
\end{equation}
the Neumann condition permits an even extension to $B_{2\rho}$, but the
extended map satisfies
\begin{equation}\label{eq:architecture-reflected-equation}
 -\eps_j\Delta v_j+\frac{a(x)}{\eps_j}\nabla V(v_j)=0,
 \quad
 a\in W^{1,\infty}(B_{2\rho}),
 \quad
 0<a_-\leq a\leq a_+.
\end{equation}
Second, one may have $\zetastar(\partial\Omega)>0$; hence the boundary part
of the limiting stress must be determined before absolute continuity of the
limiting measures is available.  The argument is divided into the following
five modules.

\paragraph{1. Variable-coefficient decay and clearing-out.}
For the energy $\cE_{\eps,a}$ defined in \eqref{eq:var-energy}, the principal
local estimate, after scaling to a sufficiently small disk, is
\begin{equation}\label{eq:architecture-clearing-out}
 \cE_{\eps,a}\left(u;B_{\frac{R}{2}}\right)
 \leq C\left[
 R^{-\frac{1}{2}}\cE_{\eps,a}(u;B_R)^{\frac{3}{2}}
 +\frac{\eps}{R}\cE_{\eps,a}(u;B_R)
 \right].
\end{equation}
To see why this is a perturbation of Bethuel's energy decay, set
\begin{equation}\label{eq:architecture-scaling}
 u_R(y):=u(p+Ry),
 \quad a_R(y):=a(p+Ry),
 \quad \eps_R:=\frac{\eps}{R}.
\end{equation}
Then
\begin{equation}\label{eq:architecture-coefficient-smallness}
 \norm{a_R-a(p)}_{L^\infty(B_1)}
 +\norm{\nabla a_R}_{L^\infty(B_1)}
 \leq 2R\norm{\nabla a}_{L^\infty(B_R(p))}.
\end{equation}
Thus the coefficient oscillation is absorbed at small scales.  The factor
$\frac{\eps}{R}=\eps_R$ in \eqref{eq:architecture-clearing-out} is required by
scaling.  This reduction applies only to the energy-decay and clearing-out
module: for
\begin{equation}\label{eq:architecture-weighted-stress}
 T_{\eps,a}
 :=e_{\eps,a}(u)I_2-\eps(\nabla u)^T\nabla u
\end{equation}
the transformed equation yields the \emph{forced stationary identity}
\begin{equation}\label{eq:architecture-forced-stationary-identity}
 \diver T_{\eps,a}
 =\frac{V(u)}{\eps}\nabla a.
\end{equation}
Hence the boundary stress and forced monotonicity arguments below cannot be
obtained by a direct invocation of Bethuel's constant-coefficient theorem.
Iteration of \eqref{eq:architecture-clearing-out}, together with the
target-space multiplier estimates, yields both clearing-out away from the wells
and comparison of the limiting full and potential energies.  After transfer
back through the boundary charts, there exist $r_c,\eta_c,C_c>0$ such that
\begin{align}
 p\in\Sstar,\quad 0<r<r_c
 &\Rightarrow
 \nustar(B_r(p)\cap\overline\Omega)\geq\eta_cr,
 \label{eq:architecture-lower-full}\\
 \nustar\left(B_{\frac{r}{2}}(p)\cap\overline\Omega\right)
 &\leq C_c\zetastar(B_r(p)\cap\overline\Omega).
 \label{eq:architecture-full-potential-comparison}
\end{align}

\paragraph{2. Boundary stress before absolute continuity.}
The diffuse stress-energy tensor is divergence-free in $\Omega$, and its Neumann
traction is normal to the boundary.  Consequently,
\begin{equation}\label{eq:architecture-inner-variation}
 \int_{\overline\Omega}T:DX=0
 \quad
 \forall X\in C^1(\overline\Omega;\R^2),
 \quad X\cdot\nOmega=0\text{ on }\partial\Omega.
\end{equation}
Testing \eqref{eq:architecture-inner-variation} with vector fields supported
in a shrinking tubular neighborhood of $\partial\Omega$ identifies the
boundary restriction of $T$ as a matrix-valued Radon measure:
\begin{equation}\label{eq:architecture-stress}
 T_{nn}\llcorner\partial\Omega
 =T_{\tau n}\llcorner\partial\Omega=0,
 \quad
 T\llcorner\partial\Omega
 =2(\tau\otimes\tau)
   (\zetastar\llcorner\partial\Omega).
\end{equation}
This step is measure-theoretic and does not assume
$\nustar,\zetastar\ll\cH^1$.

\paragraph{3. Forced radial identity in the doubled chart.}
The coefficient $a$ in \eqref{eq:architecture-reflected-equation} produces a
vector-valued force measure $h$.  For a centre $p$ on the fixed line, define
\begin{equation}\label{eq:architecture-radial-quantities}
 \begin{aligned}
 Z_p(r)&:=\zetastar^a(B_r(p)),
 &F_p(r)&:=\frac{Z_p(r)}{r},\\
 H_p(r)&:=\int_{B_r(p)}(x-p)\cdot\dd h(x),
 &\mathcal N_p&:=2\zetastar^a
 +\upmu_{\star,rr}-\upmu_{\star,\theta\theta},
\end{aligned}
\end{equation}
where the polar components are taken about $p$.  The boundary stress identity
from the preceding module supplies the missing contribution on the fixed
line and yields
\begin{equation}\label{eq:architecture-radial-positivity}
 \mathcal N_p\geq0,
 \quad
 \abs h\leq B\zetastar^a,
 \quad
 B:=\frac{\norm{\nabla a}_{L^\infty}}{a_-}.
\end{equation}
The limiting Pohozaev identity then reads
\begin{equation}\label{eq:architecture-forced-pohozaev}
 F_p(r_1)-F_p(r_0)
 =\int_{B_{r_1}(p)\backslash B_{r_0}(p)}
   \frac{1}{4\abs{x-p}}\,\dd\mathcal N_p(x)
 +\int_{r_0}^{r_1}\frac{H_p(s)}{2s^2}\,\dd s.
\end{equation}
Since
\begin{equation}\label{eq:architecture-force-bound}
 \abs{H_p(s)}
 \leq Bs\zetastar^a(B_s(p))=Bs^2F_p(s),
\end{equation}
equations \eqref{eq:architecture-radial-positivity}--
\eqref{eq:architecture-force-bound} imply, in the sense of Radon measures on
$(0,\rho)$,
\begin{equation}\label{eq:architecture-monotonicity}
 \dd F_p(r)\geq-\frac{B}{2}F_p(r)\,\dd r,
 \quad
 \dd\left(e^{\frac{Br}{2}}F_p(r)\right)\geq0.
\end{equation}
This is the boundary replacement for the exact interior monotonicity formula.

\paragraph{4. Linear density bounds and absolute continuity.}
The Gronwall inequality in \eqref{eq:architecture-monotonicity} yields the
upper bound $\zetastar(B_r(p)\cap\overline\Omega)\leq Cr$.  Combining it with
\eqref{eq:architecture-lower-full},
\eqref{eq:architecture-full-potential-comparison}, and
$0\leq\zetastar\leq\nustar$ implies
\begin{equation}\label{eq:architecture-density}
 p\in\Sstar,\quad 0<r<r_c
 \Rightarrow
 c r\leq\zetastar(B_r(p)\cap\overline\Omega)
 \leq\nustar(B_r(p)\cap\overline\Omega)
 \leq C r.
\end{equation}
The lower bound yields $\cH^1(\Sstar)<\infty$ by a $5r$-covering argument;
the upper bound yields boundary absolute continuity.  Together with the
interior theorem, this proves
\begin{equation}\label{eq:architecture-absolute-continuity}
 \Sstar=\supp\nustar=\supp\zetastar,
 \quad
 \nustar=e\,\cH^1\llcorner\Sstar,
 \quad
 \zetastar=\Uptheta\,\cH^1\llcorner\Sstar.
\end{equation}

\paragraph{5. Identification of the limiting varifold.}
The interior discrepancy identities and
\eqref{eq:architecture-stress} yield, in the approximate tangent-normal frame
$(\tau,n)$ of $\Sstar$,
\begin{equation}\label{eq:architecture-stress-identification}
 eI-m
 =\begin{pmatrix}2\Uptheta&0\\0&0\end{pmatrix}_{(\tau,n)},
 \quad
 T=2(\tau\otimes\tau)\zetastar.
\end{equation}
Substitution into \eqref{eq:architecture-inner-variation} yields
\begin{equation}\label{eq:architecture-stationarity}
\begin{aligned}
 0
 &=\frac{1}{2}\int_{\overline\Omega}T:DX
 =\int_{\Sstar}\diver_{T_x\Sstar}X\,\dd\zetastar,
 \\
 &=\delta\Vbold(\Sstar,\Uptheta)(X)
 \quad(X\cdot\nOmega=0).
\end{aligned}
\end{equation}
Finally, testing \eqref{eq:architecture-stationarity} along an isolated branch
shows that its weight is constant; a cutoff at a finite-type boundary
junction $p$ yields
\begin{equation}\label{eq:architecture-junction-balance}
 P_{T_p\partial\Omega}
 \left(\sum_\ell\Uptheta_\ell\tau_\ell\right)=0.
\end{equation}

\subsection{Organization of this paper}

The remaining sections separate the analytic estimates from their
measure-theoretic consequences.  Section~\ref{sec:preliminaries} establishes
compactness, the global range bound, the interior input, and conformal
doubling.  Section~\ref{sec:perturbative-clearing-out} proves the
variable-coefficient decay and clearing-out estimates.
Section~\ref{sec:boundary-analysis} derives the boundary stress identity,
the forced radial formula, and the two-sided density bounds.  Finally,
Section~\ref{sec:proof-main} combines these inputs to prove
Theorem~\ref{thm:main} and obtains the branch-weight and junction conclusions.

\subsection{Notations and conventions}

The following conventions are used throughout the paper.

Latin indices $i,j\in\{1,2\}$ refer to the domain, Greek indices
$\alpha,\beta\in\{1,\ldots,k\}$ refer to the target, and repeated indices
are summed.  For $A,B\in\R^{2\times2}$ we write
\begin{equation}\label{eq:intro-matrix-contraction}
 A:B:=\operatorname{tr}(A^TB)=\sum_{i,j=1}^2A_{ij}B_{ij}.
\end{equation}
For $p\in\R^2$ and $r>0$, our ball and half-ball notation is
\begin{equation}\label{eq:intro-ball-notation}
 B_r(p):=\{x\in\R^2:\abs{x-p}<r\},
 \quad B_r:=B_r(0),
 \quad B_r^+:=B_r\cap\{x_2>0\},
 \quad \Gamma_r:=B_r\cap\{x_2=0\}.
\end{equation}

\begin{itemize}
\item All balls are Euclidean.  Compact inclusion is written
$U\subset\subset\Omega$.
\item Diffuse-energy limits are Radon measures on $\overline\Omega$;
weak-star convergence is tested against $C(\overline\Omega)$, and measure
restriction is denoted by $\llcorner$.
\item At a rectifiable point, $(\tau,n)$ is an approximate tangent-normal
frame.  On $\partial\Omega$, the same letters denote the boundary tangent
and inward unit normal when no ambiguity is possible.
\item The letter $C$, with subscripts when useful, denotes a positive
constant that may change from line to line; dependencies are recorded when
they affect uniformity.
\end{itemize}

\section{Preliminaries}\label{sec:preliminaries}

This section collects the compactness, regularity, interior, and conformal
tools used in the boundary analysis.

\subsection{Potential geometry and diffuse measures}

Fix constants $\upmu_V,c_V,C_V>0$ supplied by
Assumption~\ref{assum:potential}.  The balls
$B(\upsigma_i,2\upmu_V)$ are pairwise disjoint and
\begin{equation}\label{eq:prelim-well-geometry}
 \begin{aligned}
 c_V\abs{y-\upsigma_i}^2
 &\leq V(y)\leq C_V\abs{y-\upsigma_i}^2,\\
 c_V\abs{y-\upsigma_i}^2
 &\leq\nabla V(y)\cdot(y-\upsigma_i)
 \leq C_V\abs{y-\upsigma_i}^2
\end{aligned}
\end{equation}
whenever $\abs{y-\upsigma_i}\leq2\upmu_V$.  We also choose $\upmu_V$ so
that
\begin{equation}\label{eq:prelim-radial-positivity}
 \nabla V(y)\cdot\frac{y-\upsigma_i}{\abs{y-\upsigma_i}}\geq0
 \quad\text{if }0<\abs{y-\upsigma_i}\leq2\upmu_V.
\end{equation}

Taking the trace in \eqref{eq:intro-stress-energy-tensor} and using
\eqref{eq:intro-measures}, the diffuse stress--energy tensor satisfies
\begin{equation}\label{eq:prelim-stress-trace}
 (\tr T_\eps(u_\eps))\,\dd x=2\upzeta_\eps.
\end{equation}
For the sequence in Theorem~\ref{thm:main}, abbreviate
\[
 \upnu_j:=\upnu_{\eps_j},
 \quad \upzeta_j:=\upzeta_{\eps_j},
 \quad \upmu_{j,ab}:=\upmu_{\eps_j,ab},
 \quad a,b\in\{1,2\}.
\]
We also write
\begin{equation}\label{eq:tensor-measure-matrix-notation}
 \upmu_j:=(\upmu_{j,ab})_{a,b=1}^2,
 \quad
 \upmu_\star:=(\upmu_{\star,ab})_{a,b=1}^2.
\end{equation}

\begin{prop}\label{prop:unconditional}
Under the hypotheses of Theorem~\ref{thm:main}, a subsequence satisfies
the following preliminary conclusions:
\begin{enumerate}[label=$(\theenumi)$]
\item there is a finite-perimeter partition $(E_a)_{a=1}^q$ such that
\[
 u_j\to\sum_{a=1}^q\upsigma_a\mathbf1_{E_a}
 \quad\text{in }L^1(\Omega);
\]
\item the measures $\upnu_j$, $\upzeta_j$, and $\upmu_{j,ab}$ converge weakly
star on $\overline\Omega$ to Radon measures $\nustar$, $\zetastar$, and
finite signed Radon measures $\upmu_{\star,ab}$;
\item on every $U\subset\subset\Omega$, the limits have the interior properties
stated in Theorem~\ref{thm:bethuel-interior};
\item for every $X\in C^1(\overline\Omega;\R^2)$ tangent to
$\partial\Omega$,
\begin{equation}\label{eq:inner-variation}
 \int_{\overline\Omega}\diver X\dd\nustar
 -\int_{\overline\Omega}\partial_iX_j\dd\upmu_{\star,ij}=0.
\end{equation}
\end{enumerate}
\end{prop}

\begin{proof}
We prove the four assertions in order.

\medskip
\noindent
\emph{Step 1: compactness of the phase maps and construction of the
finite-perimeter partition.}
For $y,z\in\R^k$, define the degenerate distance induced by the potential
$V$ by
\begin{equation}\label{eq:weighted-distance}
 d_V(y,z)
 :=
 \inf_{\substack{
 \gamma\in W^{1,1}([0,1];\R^k)\\
 \gamma(0)=y,\ \gamma(1)=z}}
 \int_0^1
 \sqrt{2V(\gamma(t))}\abs{\gamma'(t)}\,\dd t.
\end{equation}
Since the zero set of $V$ is
\[
 \Sigma=\{\upsigma_1,\ldots,\upsigma_q\}
\]
and the wells are isolated and non-degenerate, one has
\begin{equation}\label{eq:weighted-well-separation}
 d_\Sigma
 :=
 \min_{\alpha\neq\beta}
 d_V(\upsigma_\alpha,\upsigma_\beta)>0.
\end{equation}
For each $\alpha\in\{1,\ldots,q\}$, set
\begin{equation}\label{eq:truncated-distance-functions}
 \Phi_\alpha(y)
 :=
 \min\left\{
 d_V(y,\upsigma_\alpha),\frac{d_\Sigma}{4}
 \right\}.
\end{equation}
The function $\Phi_\alpha$ has the upper-gradient bound
\begin{equation}\label{eq:weighted-distance-gradient}
 \abs{D\Phi_\alpha(y)}
 \leq\sqrt{2V(y)}
 \quad\text{for a.e. }y\in\R^k.
\end{equation}
Consequently, the Sobolev chain rule yields
\begin{equation}\label{eq:weighted-distance-BV}
\begin{aligned}
 \abs{\nabla(\Phi_\alpha\circ u_j)}
 &\leq
 \sqrt{2V(u_j)}\abs{\nabla u_j}
 \\
 &\leq
 \frac{\eps_j}{2}\abs{\nabla u_j}^2
 +\frac{1}{\eps_j}V(u_j)
 =e_{\eps_j}(u_j)
 \quad\text{a.e. in }\Omega.
\end{aligned}
\end{equation}
Here we used
\[
 AB\leq\frac{1}{2}A^2+\frac{1}{2}B^2
\]
with
\[
 A=\sqrt{\eps_j}\abs{\nabla u_j},
 \quad
 B=\sqrt{\frac{2V(u_j)}{\eps_j}}.
\]
Since
\[
 0\leq\Phi_\alpha\leq\frac{d_\Sigma}{4},
\]
the energy bound
\[
 \cE_{\eps_j}(u_j;\Omega)\leq M_0
\]
implies
\begin{equation}\label{eq:BV-uniform-bound}
 \sum_{\alpha=1}^q
 \norm{\Phi_\alpha\circ u_j}_{BV(\Omega)}
 \leq
 \frac{q\,d_\Sigma}{4}\abs{\Omega}+qM_0.
\end{equation}
By compactness in $BV(\Omega)$, after passing to a common subsequence there
exist $\phi_\alpha\in BV(\Omega)$ such that
\begin{equation}\label{eq:phase-functions-BV-convergence}
 \Phi_\alpha\circ u_j\to\phi_\alpha
 \quad\text{strongly in }L^1(\Omega)
 \quad\text{and a.e. in }\Omega
\end{equation}
for every $\alpha$.

On the other hand,
\begin{equation}\label{eq:potential-goes-to-zero}
 \int_\Omega V(u_j)\,\dd x
 \leq
 \eps_j\cE_{\eps_j}(u_j;\Omega)
 \leq\eps_jM_0
 \to0.
\end{equation}
For every $\delta>0$, the assumptions on $V$ imply
\[
 c_\delta
 :=
 \inf\left\{
 V(y):\dist(y,\Sigma)\geq\delta
 \right\}>0.
\]
Therefore
\begin{equation}\label{eq:distance-to-wells-in-measure}
 \begin{aligned}
 \mathcal L^2\left(
 \{x\in\Omega:\dist(u_j(x),\Sigma)\geq\delta\}
 \right)
 &\leq
 \frac{1}{c_\delta}\int_\Omega V(u_j)\,\dd x\\
 &\to0.
\end{aligned}
\end{equation}
After passing to a further subsequence,
\begin{equation}\label{eq:distance-to-wells-ae}
 \dist(u_j(x),\Sigma)\to0
 \quad\text{for a.e. }x\in\Omega.
\end{equation}

The vectors
\[
 \left(\Phi_1(\upsigma_\alpha),\ldots,
       \Phi_q(\upsigma_\alpha)\right),
 \quad \alpha=1,\ldots,q,
\]
are pairwise distinct, because
\begin{equation}\label{eq:well-signatures}
 \Phi_\alpha(\upsigma_\alpha)=0,
 \quad
 \Phi_\beta(\upsigma_\alpha)=\frac{d_\Sigma}{4}
 \quad\text{if }\beta\neq\alpha.
\end{equation}
It follows from
\eqref{eq:phase-functions-BV-convergence} and
\eqref{eq:distance-to-wells-ae} that, for a.e. $x\in\Omega$, there is a
unique index $\alpha(x)\in\{1,\ldots,q\}$ such that
\[
 u_j(x)\to\upsigma_{\alpha(x)}.
\]
Define
\begin{equation}\label{eq:phase-sets-definition}
 E_\alpha
 :=
 \{x\in\Omega:\alpha(x)=\alpha\},
 \quad
 u_0
 :=
 \sum_{\alpha=1}^q
 \upsigma_\alpha\mathbf1_{E_\alpha}.
\end{equation}
Then $(E_\alpha)_{\alpha=1}^q$ is a measurable partition of $\Omega$ up to
a set of $\mathcal L^2$-measure zero, and
\begin{equation}\label{eq:phase-partition-convergence-ae}
 u_j\to u_0
 \quad\text{a.e. in }\Omega.
\end{equation}

Moreover, \eqref{eq:well-signatures} yields
\begin{equation}\label{eq:limit-distance-functions}
 \phi_\alpha
 =
 \Phi_\alpha(u_0)
 =
 \frac{d_\Sigma}{4}
 \left(1-\mathbf1_{E_\alpha}\right)
 \quad\text{a.e. in }\Omega.
\end{equation}
Hence $\mathbf1_{E_\alpha}\in BV(\Omega)$ and
\begin{equation}\label{eq:partition-perimeter-estimate}
\begin{aligned}
 \frac{d_\Sigma}{4}
 \sum_{\alpha=1}^q\Per(E_\alpha;\Omega)
 &=
 \sum_{\alpha=1}^q
 \abs{D\phi_\alpha}(\Omega)
 \\
 &\leq
 \liminf_{j\to\infty}
 \sum_{\alpha=1}^q
 \int_\Omega
 \abs{\nabla(\Phi_\alpha\circ u_j)}\,\dd x
 \\
 &\leq qM_0.
\end{aligned}
\end{equation}
Thus
\begin{equation}\label{eq:partition-perimeter-final}
 \sum_{\alpha=1}^q\Per(E_\alpha;\Omega)
 \leq\frac{4q}{d_\Sigma}M_0.
\end{equation}

Finally, Assumption~\ref{assum:potential}(iii) provides constants
$c_V^\infty>0$ and $C_V^\infty<+\infty$ such that
\begin{equation}\label{eq:potential-coercive-global}
 V(y)\geq c_V^\infty\abs y^2-C_V^\infty
 \quad\forall y\in\R^k.
\end{equation}
Consequently,
\begin{equation}\label{eq:L2-uniform-bound}
\begin{aligned}
 c_V^\infty\int_\Omega\abs{u_j}^2\,\dd x
 &\leq
 \int_\Omega V(u_j)\,\dd x
 +C_V^\infty\abs{\Omega}
 \\
 &\leq
 \eps_jM_0+C_V^\infty\abs{\Omega},
\end{aligned}
\end{equation}
so $(u_j)$ is uniformly integrable in $L^1(\Omega)$.  Combining this with
\eqref{eq:phase-partition-convergence-ae} and Vitali's theorem yields
\begin{equation}\label{eq:phase-partition-L1-convergence}
 u_j\to
 \sum_{\alpha=1}^q
 \upsigma_\alpha\mathbf1_{E_\alpha}
 \quad\text{strongly in }L^1(\Omega).
\end{equation}
This proves assertion $(1)$.

\medskip
\noindent
\emph{Step 2: weak-star compactness of the energy measures.}
Recall that
\begin{align}
 \upnu_j
 &:=
 \left(
 \frac{\eps_j}{2}\abs{\nabla u_j}^2
 +\frac{1}{\eps_j}V(u_j)
 \right)\cL^2\llcorner\Omega,
 \label{eq:full-energy-measure-proof}\\
 \upzeta_j
 &:=
 \frac{1}{\eps_j}V(u_j)\,
 \cL^2\llcorner\Omega,
 \label{eq:potential-energy-measure-proof}\\
 \upmu_{j,i\ell}
 &:=
 \eps_j\,
 \partial_i u_j\cdot\partial_\ell u_j\,
 \cL^2\llcorner\Omega,
 \quad i,\ell\in\{1,2\}.
 \label{eq:gradient-tensor-measure-proof}
\end{align}
We regard these as measures on $\overline\Omega$.  The energy bound implies
\begin{equation}\label{eq:measure-mass-bounds}
 \upnu_j(\overline\Omega)\leq M_0,
 \quad
 0\leq\upzeta_j\leq\upnu_j.
\end{equation}
Furthermore,
\begin{equation}\label{eq:gradient-measure-bound}
\begin{aligned}
 \abs{\upmu_{j,i\ell}}
 &\leq
 \frac{\eps_j}{2}
 \left(
 \abs{\partial_i u_j}^2+
 \abs{\partial_\ell u_j}^2
 \right)\cL^2
 \\
 &\leq
 \eps_j\abs{\nabla u_j}^2\cL^2
 \leq2\upnu_j.
\end{aligned}
\end{equation}
The sequential weak-star compactness of bounded subsets of
$\cM(\overline\Omega)$ therefore yields, after passing to a subsequence,
\begin{equation}\label{eq:measure-weak-star-convergence}
 \begin{aligned}
  \upnu_j&\stackrel{*}{\rightharpoonup}\nustar,
  &\upzeta_j&\stackrel{*}{\rightharpoonup}\zetastar,\\
  \upmu_{j,i\ell}&\stackrel{*}{\rightharpoonup}\upmu_{\star,i\ell}
  &&\text{in }\cM(\overline\Omega).
\end{aligned}
\end{equation}
Here $\nustar$ and $\zetastar$ are non-negative Radon measures, while each
$\upmu_{\star,i\ell}$ is a finite signed Radon measure.  This proves
assertion $(2)$.

\medskip
\noindent
\emph{Step 3: interior conclusions.}
Choose an exhaustion of $\Omega$ by bounded smooth open sets
\begin{equation}\label{eq:interior-exhaustion}
 U_1\subset\subset U_2\subset\subset\cdots\subset\subset\Omega,
 \quad
 \bigcup_{m=1}^\infty U_m=\Omega.
\end{equation}
For every fixed $m$, the restrictions $u_j|_{U_m}$ satisfy the hypotheses
of Theorem~\ref{thm:bethuel-interior}.  Applying that theorem successively
on $U_m$ and using a diagonal subsequence, we obtain all the conclusions of
Theorem~\ref{thm:bethuel-interior} simultaneously on every
$U\subset\subset\Omega$.  In particular, the interior rectifiability,
density, discrepancy, and stress--energy identities hold locally in
$\Omega$.  This proves assertion $(3)$.

\medskip
\noindent
\emph{Step 4: passage of the stationary identity to the limit.}
Lemma~\ref{lem:global-neumann-Linf} and the boundary regularity statement
following it apply independently to each $u_j$ and justify the classical
identities used in this step.
Define the diffuse stress--energy tensor
\begin{equation}\label{eq:diffuse-stress-proof}
 T_{\eps_j}(u_j)
 :=
 e_{\eps_j}(u_j)I_2
 -\eps_j(\nabla u_j)^T\nabla u_j.
\end{equation}
In components,
\[
 (T_{\eps_j}(u_j))_{i\ell}
 =
 e_{\eps_j}(u_j)\delta_{i\ell}
 -\eps_j\partial_i u_j\cdot\partial_\ell u_j.
\]
A direct calculation using the Euler--Lagrange equation shows
\begin{equation}\label{eq:diffuse-stress-divergence}
\begin{aligned}
 \partial_\ell(T_{\eps_j}(u_j))_{i\ell}
 &=
 \left(
 -\eps_j\Delta u_j
 +\frac{1}{\eps_j}\nabla V(u_j)
 \right)\cdot\partial_i u_j
 =0
 \quad\text{in }\Omega.
\end{aligned}
\end{equation}
Thus $\diver T_{\eps_j}(u_j)=0$ in the distributional sense.  On
$\partial\Omega$, the Neumann condition implies
\begin{equation}\label{eq:diffuse-stress-traction}
\begin{aligned}
 (T_{\eps_j}(u_j)\nOmega)_i
 &=
 e_{\eps_j}(u_j)(\nOmega)_i
 -\eps_j\partial_i u_j\cdot
   \partial_{\nOmega}u_j
 \\
 &=
 e_{\eps_j}(u_j)(\nOmega)_i.
\end{aligned}
\end{equation}
Let $X\in C^1(\overline\Omega;\R^2)$ satisfy
\[
 X\cdot\nOmega=0
 \quad\text{on }\partial\Omega.
\]
Then \eqref{eq:diffuse-stress-traction} implies
\[
 (T_{\eps_j}(u_j)\nOmega)\cdot X=0
 \quad\text{on }\partial\Omega.
\]
Hence the divergence theorem and
\eqref{eq:diffuse-stress-divergence} yield
\begin{equation}\label{eq:diffuse-inner-variation}
\begin{aligned}
 0
 &=
 \int_{\partial\Omega}
 (T_{\eps_j}(u_j)\nOmega)\cdot X\,\dd\cH^1
 \\
 &=
 \int_\Omega
 \diver\!\left(T_{\eps_j}(u_j)X\right)\,\dd x
 \\
 &=
 \int_\Omega
 T_{\eps_j}(u_j):DX\,\dd x
 \\
 &=
 \int_{\overline\Omega}
 \diver X\,\dd\upnu_j
 -
 \int_{\overline\Omega}
 \partial_iX_\ell\,\dd\upmu_{j,i\ell}.
\end{aligned}
\end{equation}
Since $\diver X$ and $\partial_iX_\ell$ are continuous on
$\overline\Omega$, the convergences
\eqref{eq:measure-weak-star-convergence} allow us to pass to the limit:
\begin{equation}\label{eq:limiting-inner-variation-proof}
 \int_{\overline\Omega}\diver X\,\dd\nustar
 -
 \int_{\overline\Omega}
 \partial_iX_\ell\,\dd\upmu_{\star,i\ell}
 =0.
\end{equation}
This is precisely \eqref{eq:inner-variation} and proves assertion $(4)$.
\end{proof}

\subsection{A global range bound}

The coercivity of the potential yields a uniform range bound for every
Neumann solution.

\begin{lem}\label{lem:global-neumann-Linf}
There is $L_V<\infty$, depending only on $V$, such that every solution of
\eqref{eq:intro-system} satisfies
$\norm{u_\eps}_{L^\infty(\Omega)}\leq L_V$.
\end{lem}

\begin{proof}
Set $R_V:=\max\{R_\infty,1\}$ and define the non-decreasing Lipschitz
function
\[
 g(t):=\min\{(t-R_V)_+,1\}.
\]
The bounded map $F:\R^k\to\R^k$ given by
\[
 F(y):=
 \begin{cases}
 g(\abs y)\frac{y}{\abs y},&y\neq0,\\
 0,&y=0,
 \end{cases}
\]
is globally Lipschitz because it vanishes on $B_{R_V}$.  Hence
$\varphi_\eps:=F(u_\eps)$ belongs to
$H^1(\Omega;\R^k)\cap L^\infty(\Omega;\R^k)$ and is admissible in
\eqref{eq:intro-weak-system}.  At almost every point with
$\abs{u_\eps}>R_V$,
\begin{equation}\label{eq:range-radial-coercivity}
 \begin{aligned}
 \partial_i u_\eps\cdot\partial_iF(u_\eps)
 &=
 g'(\abs{u_\eps})
 \abs{\partial_i\abs{u_\eps}}^2\\
 &\quad+
 \frac{g(\abs{u_\eps})}{\abs{u_\eps}}
 \left(
 \abs{\partial_i u_\eps}^2
 -\abs{\partial_i\abs{u_\eps}}^2
 \right)
 \geq0.
 \end{aligned}
\end{equation}
Assumption~\ref{assum:potential}(iii) also yields
\[
 \nabla V(u_\eps)\cdot F(u_\eps)
 =
 \frac{g(\abs{u_\eps})}{\abs{u_\eps}}
 u_\eps\cdot\nabla V(u_\eps)
 \geq
 \alpha_\infty g(\abs{u_\eps})\abs{u_\eps}
\]
on the same set.  Testing with $\varphi_\eps$ therefore yields
\begin{equation}\label{eq:range-truncation-test}
 0
 \geq
 \frac{\alpha_\infty}{\eps}
 \int_{\{\abs{u_\eps}>R_V\}}
 g(\abs{u_\eps})\abs{u_\eps}\,\dd x
 \geq0.
\end{equation}
The middle integral can vanish only if
$\abs{u_\eps}\leq R_V$ almost everywhere.  This proves the claim with
$L_V:=R_V$ without assuming prior regularity of $u_\eps$.
\end{proof}

The range bound makes $\nabla V(u_\eps)$ bounded.  Boundary $W^{2,p}$
regularity for the Neumann problem, followed by Schauder estimates, implies
$u_\eps\in C^{2,\alpha}(\overline\Omega;\R^k)$ for every $0<\alpha<1$, which
justifies the classical differential identities used below.

\subsection{The interior vectorial theorem}

We record the part of Bethuel's theorem used below.

\begin{thm}[Bethuel's interior theorem]
\label{thm:bethuel-interior}
Let $D\subset\subset\R^2$ be a bounded smooth domain, and let $v_j$ solve
\begin{equation}\label{eq:bethuel-interior-equation}
 -\eps_j\Delta v_j+\frac{1}{\eps_j}\nabla V(v_j)=0
 \quad\text{in }D,
\end{equation}
where $\eps_j\to0$ and
\[
 \cE_{\eps_j}(v_j;D)\leq M.
\]
After passing to a subsequence, there exist a closed countably
$1$-rectifiable set $S_D\subset D$ and Radon measures
$\upnu_D$, $\upzeta_D$, and $\upmu_{D,i\ell}$,
$i,\ell\in\{1,2\}$, such that
\begin{equation}\label{eq:bethuel-measure-convergence}
 \begin{aligned}
 \left[\frac{\eps_j}{2}\abs{\nabla v_j}^2
       +\frac{1}{\eps_j}V(v_j)\right]\dd x
 &\stackrel{*}{\rightharpoonup}\upnu_D,\\
 \frac{1}{\eps_j}V(v_j)\dd x
 &\stackrel{*}{\rightharpoonup}\upzeta_D,\\
 \eps_j\,\partial_i v_j\cdot\partial_\ell v_j\dd x
 &\stackrel{*}{\rightharpoonup}\upmu_{D,i\ell}
\end{aligned}
\end{equation}
as Radon measures locally in $D$, and the following properties hold.

\begin{enumerate}[label=$(\theenumi)$]
\item There exists a locally constant map
\[
 v_\star:D\backslash S_D\to\Sigma
\]
such that
\[
 v_j\to v_\star
 \quad\text{locally uniformly in }D\backslash S_D.
\]

\item The limiting measures are concentrated on $S_D$ and admit the
representations
\begin{equation}\label{eq:bethuel-density-representations}
 \upnu_D=e_D\,\cH^1\llcorner S_D,
 \quad
 \upzeta_D=\Uptheta_D\,\cH^1\llcorner S_D,
 \quad
 \upmu_{D,i\ell}
 =m_{D,i\ell}\,\cH^1\llcorner S_D.
\end{equation}
For every $K\subset\subset D$, there exist constants
$0<c_K\leq C_K<+\infty$ such that
\begin{equation}\label{eq:bethuel-density-bounds}
 c_K\leq e_D,\Uptheta_D\leq C_K,
 \quad
 \abs{m_{D,i\ell}}\leq C_K
\end{equation}
for $\cH^1$-almost every point of $S_D\cap K$.

\item At $\cH^1$-almost every $x\in S_D$, let $\tau(x)$ be a unit
vector spanning the approximate tangent line $T_xS_D$, and set
$n(x):=\tau(x)^\perp$.  In the frame $(\tau,n)$, the limiting densities
satisfy the vectorial discrepancy relations
\begin{equation}\label{eq:bethuel-discrepancy-relations}
 e_D=m_{D,nn},
 \quad
 2\Uptheta_D=m_{D,nn}-m_{D,\tau\tau},
 \quad
 m_{D,n\tau}=0.
\end{equation}
Equivalently, the limiting stress--energy tensor
\begin{equation}\label{eq:bethuel-limiting-stress}
 T_D:=\upnu_D I_2-\upmu_D
\end{equation}
has the representation
\begin{equation}\label{eq:bethuel-stress-representation}
 T_D=2(\tau\otimes\tau)\upzeta_D.
\end{equation}

\item The rectifiable varifold determined by the potential-energy measure,
\[
 \mathbf V_D:=\Vbold(S_D,\Uptheta_D),
 \quad
 \norm{\mathbf V_D}=\upzeta_D,
\]
is stationary in $D$:
\begin{equation}\label{eq:bethuel-varifold-stationarity}
 \delta\mathbf V_D(X)
 =
 \int_{S_D}
 (\tau\otimes\tau):DX\,
 \Uptheta_D\,\dd\cH^1
 =0
 \quad
 \forall X\in C_c^1(D;\R^2).
\end{equation}

\item There exists a set $N_D\subset S_D$ with
$\cH^1(N_D)=0$ such that, for every $x\in S_D\backslash N_D$, there are
$r_x>0$, a line $L_x$ through $x$, and a constant $\Uptheta_x>0$ for
which
\begin{equation}\label{eq:bethuel-local-line-structure}
 \cH^1\left(
 (S_D\mathbin{\triangle}L_x)\cap B_{r_x}(x)
 \right)=0,
 \quad
 \upzeta_D\llcorner B_{r_x}(x)
 =
 \Uptheta_x\,
 \cH^1\llcorner
 \left(L_x\cap B_{r_x}(x)\right).
\end{equation}
\end{enumerate}
\end{thm}

\begin{proof}
Assertion~$(1)$ and the rectifiability of $S_D$ are
\cite[Theorem~1.2]{Bethuel2025}.  The representations and bounds in
assertion~$(2)$ follow from \cite[Theorem~1.8]{Bethuel2025}, together with
the domination of the gradient measures by the full energy measure.
The two identities
\[
 2\Uptheta_D=m_{D,nn}-m_{D,\tau\tau},
 \quad m_{D,n\tau}=0
\]
are \cite[Theorem~1.9]{Bethuel2025}; combining them with
$e_D=\frac{1}{2}\tr m_D+\Uptheta_D$ yields $e_D=m_{D,nn}$ and hence
\eqref{eq:bethuel-stress-representation}.  Assertions~$(4)$ and~$(5)$ are
\cite[Theorems~1.5 and~1.3]{Bethuel2025}, respectively.
\end{proof}

\subsection{Conformal charts and doubled measures}

Fix $p\in\partial\Omega$.  Choose a $C^3$ Jordan subdomain
$\mathcal O_p\subset\Omega$ whose boundary agrees with
$\partial\Omega$ near $p$.  Since a $C^3$ curve is of class
$C^{2,\alpha}$ for every $0<\alpha<1$, the Kellogg--Warschawski theorem
\cite{Warschawski1935}\cite[Theorems 3.5 and 3.6]{Pommerenke1992}, applied
to $\mathcal O_p$, provides a
conformal boundary chart
\[
 \Psi:B_{2\rho}^+\to\mathcal O_p
\]
whose flat edge maps to $\partial\Omega$ near $p$.  The map extends as a
$C^2$ map with non-vanishing derivative to that edge.
After shrinking the chart, $\Psi$ is uniformly bi-Lipschitz.  Conformality
means that, for a rotation $Q(x)\in SO(2)$,
\begin{equation}\label{eq:conformal-differential}
 D\Psi(x)=\sqrt{a(x)}Q(x),
 \quad
 a(x):=\abs{\Psi'(x)}^2=\det D\Psi(x)>0.
\end{equation}
If $v:=u\circ\Psi$, then
\begin{align}
 \Delta v(x)&=a(x)(\Delta u)(\Psi(x)),
 \label{eq:conformal-laplacian}\\
 \int_{\Psi(U)}\abs{\nabla u}^2\,\dd y
 &=\int_U\abs{\nabla v}^2\,\dd x,
 \quad
 \int_{\Psi(U)}V(u)\,\dd y
 =\int_Ua(x)V(v)\,\dd x.
 \label{eq:conformal-energy-change}
\end{align}
Consequently,
\begin{equation}\label{eq:conformal-system}
 -\eps\Delta v+\frac{a}{\eps}\nabla V(v)=0,
 \quad
 \partial_2v=0\quad\text{on }\{x_2=0\}.
\end{equation}
The coefficient stress in chart coordinates is related to the physical
stress by
\begin{equation}\label{eq:conformal-stress-change}
 T_{\eps,a}(v)
 =aQ^T\left(T_\eps(u)\circ\Psi\right)Q,
 \quad D\Psi=\sqrt a\,Q,
 \quad Q\in SO(2).
\end{equation}
Thus the tangential and normal components of the boundary stress identity
are preserved, up to the rotation $Q$, under the conformal change of
coordinates.

\begin{lem}\label{lem:measure-double}
Let $R(x_1,x_2):=(x_1,-x_2)$ and let $\lambda_j$ be diffuse Radon
measures on the upper half-ball with weak-star limit on the closed half-ball
$\lambda=\lambda^\circ+\lambda^\partial$, where
\begin{equation}\label{eq:open-boundary-decomposition}
 \lambda^\circ:=\lambda\llcorner B_{2\rho}^+,
 \quad
 \lambda^\partial:=\lambda\llcorner\Gamma_{2\rho}.
\end{equation}
If
$\widetilde\lambda_j:=\lambda_j+R_\#\lambda_j$, then
\begin{equation}\label{eq:measure-double-convergence}
 \widetilde\lambda_j\stackrel{*}{\rightharpoonup}\widetilde\lambda
 \quad\text{on }B_{2\rho},
\end{equation}
where
\begin{equation}\label{eq:measure-double}
 \widetilde\lambda
 =\lambda^\circ+R_\#\lambda^\circ+2\lambda^\partial.
\end{equation}
For a symmetric matrix-valued measure
$\Lambda_j=(\lambda_{j,ab})_{a,b=1}^2$ with limit
$\Lambda=\Lambda^\circ+\Lambda^\partial$, define
$\widetilde\Lambda_j:=\Lambda_j+Q(R_\#\Lambda_j)Q$, where
$Q=\operatorname{diag}(1,-1)$.  Then
\begin{equation}\label{eq:tensor-measure-double}
 \widetilde\Lambda_j\stackrel{*}{\rightharpoonup}
 \Lambda^\circ+Q(R_\#\Lambda^\circ)Q
 +\Lambda^\partial+Q\Lambda^\partial Q.
\end{equation}
Thus the diagonal fixed-line components are doubled and the mixed
fixed-line components cancel.
\end{lem}

\begin{proof}
For every $\phi\in C_c(B_{2\rho})$,
\begin{equation}\label{eq:measure-double-test}
\begin{aligned}
 \int\phi\,\dd\widetilde\lambda_j
 &=\int_{B_{2\rho}^+}(\phi+\phi\circ R)\,\dd\lambda_j
 \\
 &\to
 \int_{B_{2\rho}^+}(\phi+\phi\circ R)\,\dd\lambda^\circ
 +2\int_{\Gamma_{2\rho}}\phi\,\dd\lambda^\partial
 \\
 &=\int\phi\,\dd
 \left(\lambda^\circ+R_\#\lambda^\circ+2\lambda^\partial\right).
\end{aligned}
\end{equation}
If $\widetilde u=u\circ R$ in the lower half-ball, then
\begin{equation}\label{eq:reflected-gradient-tensor}
 \nabla\widetilde u(x)=\nabla u(Rx)Q,
 \quad
 (\nabla\widetilde u)^T\nabla\widetilde u
 =Q\left[(\nabla u)^T\nabla u\right](Rx)Q,
\end{equation}
which proves the tensor formula.
\end{proof}
\begin{rem}
Lemma~\ref{lem:measure-double} will be applied to the full- and
potential-energy measures $\upnu_j$ and $\upzeta_j$, and in its
matrix-valued form to the gradient measures $\upmu_j$.  It records that
doubling reflects the interior mass, while limiting mass supported on
the fixed line $\Gamma_{2\rho}$ is counted twice.  For tensor measures, the
conjugation by $Q$ accounts for the sign change of the mixed
tangential--normal components.
\end{rem}

\section{Perturbative clearing-out for the reflected equation}
\label{sec:perturbative-clearing-out}
\label{sec:var-clearing-out}

Let $V$ satisfy Assumption~\ref{assum:potential}.  Fix
$\upmu_V>0$ and $0<c_V\leq C_V$ so that the balls
$B(\upsigma_i,2\upmu_V)$ are pairwise disjoint and
\begin{equation}\label{eq:var-well}
\begin{aligned}
c_V\abs{y-\upsigma_i}^2
&\leq V(y)\leq C_V\abs{y-\upsigma_i}^2,\\
c_V\abs{y-\upsigma_i}^2
&\leq \nabla V(y)\cdot(y-\upsigma_i)
 \leq C_V\abs{y-\upsigma_i}^2.
\end{aligned}
\end{equation}
These inequalities hold when $\abs{y-\upsigma_i}\leq2\upmu_V$.  Decreasing
$\upmu_V$ if needed, we also require
that
\begin{equation}\label{eq:var-radial}
\nabla V(y)\cdot \frac{y-\upsigma_i}{\abs{y-\upsigma_i}}\geq0
\quad\text{if }0<\abs{y-\upsigma_i}\leq2\upmu_V.
\end{equation}

Fix $0<a_-\leq a_+<\infty$.  For
$a\in W^{1,\infty}(B_1)$ satisfying $a_-\leq a\leq a_+$, set
\begin{equation}\label{eq:var-energy}
 \begin{aligned}
 e_{\eps,a}(u)(x)
 &:=\frac{\eps}{2}\abs{\nabla u(x)}^2
 +\frac{a(x)}{\eps}V(u(x)),\\
 \cE_{\eps,a}(u;U)&:=\int_Ue_{\eps,a}(u)\,\dd x,\\
 \mathbb V_{\eps,a}(u;U)
 &:=\int_U\frac{a(x)}{\eps}V(u)\,\dd x.
\end{aligned}
\end{equation}
We consider
\begin{equation}\label{eq:var-pde}
-\eps\Delta u+\frac{a(x)}{\eps}\nabla V(u)=0
\quad\text{in }B_1.
\end{equation}

\begin{thm}[Variable-coefficient decay and clearing-out]
\label{thm:var-clearing-out}
There are
\[
\delta_0=\frac{a_-}{2},\quad
\eta_0=\eta_0(V,a_-,a_+)>0,\quad
C_0=C_0(V,a_-,a_+)<\infty
\]
such that, if $\norm{\nabla a}_{L^\infty}\leq\delta_0$,
$0<\eps\leq1$, and $u\in H^1(B_1;\R^k)$ is a weak solution of
\eqref{eq:var-pde}, then
\begin{equation}\label{eq:var-decay}
\cE_{\eps,a}\left(u;B_{\frac{1}{2}}\right)
\le C_0\{\cE_{\eps,a}(u;B_1)^{\frac{3}{2}}
          +\eps \cE_{\eps,a}(u;B_1)\}.
\end{equation}
If $\cE_{\eps,a}(u;B_1)\le2\eta_0$, then one
$\upsigma\in\Sigma$ satisfies
\begin{equation}\label{eq:var-confine}
\abs{u-\upsigma}\leq\frac{\upmu_V}{2}\quad\text{on }B_{\frac{3}{4}},
\end{equation}
and
\begin{equation}\label{eq:var-small-inner}
\cE_{\eps,a}\left(u;B_{\frac{5}{8}}\right)
\le C_0\eps \cE_{\eps,a}(u;B_1).
\end{equation}
The constants are uniform over this coefficient class.
\end{thm}
\begin{rem}
Estimate~\eqref{eq:var-decay} extends Bethuel's Proposition~1.12 to variable
coefficients.  The confinement and inner-energy estimates
\eqref{eq:var-confine}--\eqref{eq:var-small-inner} yield the corresponding,
slightly weaker, form of Bethuel's Theorem~1.11; the latter also records the
quantitative bound
$\abs{u-\upsigma}\leq C\cE_\eps(u;B_1)^{\frac{1}{6}}$.
\end{rem}

\subsection{Preliminary estimates}

We first record the circle, range, and far-from-well estimates used in the
decay argument.

Choose a non-decreasing Lipschitz function
$\psi:[0,\infty)\to\left[0,\frac{3\upmu_V}{4}\right]$, smooth on
$(0,\infty)$, such that
\[
 \psi(t)=t\quad\text{for }t\leq\frac{\upmu_V}{2},
 \quad
 \psi(t)\geq\frac{\upmu_V}{2}\quad\text{for }t\geq\frac{\upmu_V}{2},
 \quad
 \psi(t)=\frac{3\upmu_V}{4}\quad\text{for }t\geq\upmu_V,
\]
with $0\leq\psi'\leq1$, and require $\psi^2$ to be smooth at zero.  Set
$\chi_i(y):=\psi(\abs{y-\upsigma_i})$.  Thus
$\chi_i<\frac{\upmu_V}{2}$ forces
$\abs{y-\upsigma_i}<\frac{\upmu_V}{2}$.  All level values
used below are strictly positive, where $\chi_i$ is smooth; the Lipschitz
coarea formula provides the same conclusion directly.
Put $w_i:=\chi_i\circ u$ and $J(u):=\abs{\nabla u}\sqrt{V(u)}$.
Young's inequality and \eqref{eq:var-well} yield
\begin{equation}\label{eq:var-MM}
\sqrt{2a_-}\,J(u)\leq e_{\eps,a}(u),\quad
\abs{\nabla(w_i^2)}\leq C(V,a_-)e_{\eps,a}(u).
\end{equation}

\begin{lem}\label{lem:var-circle}
There are $h_0,C_{\rm cir}>0$, depending only on
$V,a_-,a_+$, such that, if $r\in[\frac{1}{2},1)$ and
\[
O(r):=\int_{\partial B_r}e_{\eps,a}(u)\le h_0,
\]
then some $\upsigma\in\Sigma$ satisfies
\[
\sup_{\partial B_r}\abs{u-\upsigma}\leq C_{\rm cir}\sqrt{O(r)}.
\]
Also, if $0\leq r_0<r_1\leq1$, some $r\in[r_0,r_1]$ satisfies
\[
O(r)\le \frac{\cE_{\eps,a}(u;B_{r_1}\backslash B_{r_0})}{r_1-r_0}.
\]
\end{lem}

\noindent\emph{Bethuel correspondence.}
The first assertion is the variable-coefficient version of
\cite[Lemma 2.6]{Bethuel2025}; the selection of a good radius is the coarea
step used in \cite[Lemma 2.7]{Bethuel2025}.  The only modification is the
uniform comparison $a_-V\leq aV\leq a_+V$.

\begin{proof}
The coarea identity ensures
\begin{equation}\label{eq:circle-Fubini}
 \int_{r_0}^{r_1}O(s)\,\dd s
 =\cE_{\eps,a}(u;B_{r_1}\backslash B_{r_0}),
\end{equation}
which proves the second assertion.  Since $r\geq\frac{1}{2}$,
\begin{equation}\label{eq:circle-base-point}
 \min_{x\in\partial B_r}V(u(x))
 \leq\frac{1}{2\pi r}\int_{\partial B_r}V(u)\,\dd\cH^1
 \leq\frac{\eps}{\pi a_-}O(r)\leq\frac{1}{\pi a_-}O(r).
\end{equation}
For $O(r)\leq h_0$ and $h_0$ below the fixed well-separation threshold,
there exist $x_r\in\partial B_r$ and a unique $\upsigma\in\Sigma$ such that
\begin{equation}\label{eq:circle-base-well}
 \abs{u(x_r)-\upsigma}\leq C(V,a_-)O(r)^{\frac{1}{2}}.
\end{equation}
For every $x\in\partial B_r$, integration along either circular arc from
$x_r$ to $x$ and \eqref{eq:var-MM} yield
\begin{equation}\label{eq:circle-oscillation}
\begin{aligned}
 \abs{w_\upsigma(x)^2-w_\upsigma(x_r)^2}
 &\leq\int_{\partial B_r}\abs{\nabla(w_\upsigma^2)}\,\dd\cH^1
 \leq C(V,a_-)O(r).
\end{aligned}
\end{equation}
Choose $h_0$ so that
\[
 \abs{u(x_r)-\upsigma}<\frac{\upmu_V}{4},
 \quad
 \sup_{x\in\partial B_r}
 \abs{w_\upsigma(x)^2-w_\upsigma(x_r)^2}
 <\frac{\upmu_V^2}{16}.
\]
Then $w_\upsigma(x_r)=\abs{u(x_r)-\upsigma}$ and
$w_\upsigma^2<\frac{\upmu_V^2}{8}$ throughout $\partial B_r$.
Consequently, $w_\upsigma=\abs{u-\upsigma}$ on the entire circle, and
\begin{equation}\label{eq:circle-final}
 \sup_{\partial B_r}\abs{u-\upsigma}^2
 \leq C(V,a_-)O(r).
\end{equation}
\end{proof}

\begin{lem}\label{lem:var-far}
If $\cE_{\eps,a}(u;B_1)\leq M$, then
\[
 \norm{u}_{L^\infty(B_{\frac{7}{8}})}
 +\eps\norm{\nabla u}_{L^\infty(B_{\frac{4}{5}})}
 \leq C(M,V,a_-,a_+).
\]
On
\[
 \Xi:=\left\{x\in B_{\frac{4}{5}}:\dist(u(x),\Sigma)
 \geq\frac{\upmu_V}{4}\right\},
\]
one has
\begin{equation}\label{eq:var-far}
 e_{\eps,a}(u)
 \leq C(M,V,a_-,a_+)\frac{a(x)}{\eps}V(u).
\end{equation}
\end{lem}

\begin{proof}
Assumption~\ref{assum:potential}(iii) implies
$y\cdot\nabla V(y)\geq c\abs{y}^2-C$ and
$V(y)\geq c\abs{y}^2-C$.  Hence
\[
 -\Delta\abs{u}^2+c_1\eps^{-2}\abs{u}^2\leq C_1\eps^{-2},
 \quad
 \int_{B_1}\abs{u}^2\dd x\leq C(1+M).
\]
The positive part of $\abs{u}^2-\frac{C_1}{c_1}$ is subharmonic, so the
mean-value inequality yields the $L^\infty$ estimate.  Rescaling by
$\eps$, interior $W^{2,p}$ estimates and Sobolev embedding imply
$\eps\abs{\nabla u}\leq C$.  On $\Xi$, boundedness of $u$ and finiteness
of $\Sigma$ imply $V(u)\geq c(M,V)>0$, which proves
\eqref{eq:var-far}.
\end{proof}

\subsection{Target-space multipliers}

Target-space truncations compare the energy near the wells with the
potential energy and a boundary error.

For $s\leq\frac{\upmu_V}{2}$, set
\[
 \Omega_i(s,\rho):=\{x\in B_\rho:\abs{u-\upsigma_i}<s\},
 \quad
 \Upsilon(s,\rho):=\bigcup_i\Omega_i(s,\rho).
\]
At a regular value, write
\[
 \Gamma_i(s,\rho)
 :=\partial\Omega_i(s,\rho)\cap B_\rho,
 \quad
 \Pi_i(s,\rho)
 :=\partial\Omega_i(s,\rho)\cap\partial B_\rho.
\]
On $\Gamma_i(s,\rho)$ let
\begin{equation}\label{eq:level-normal-definition}
 n_i:=\frac{\nabla\abs{u-\upsigma_i}}{\abs{\nabla\abs{u-\upsigma_i}}},
\end{equation}
the outer unit normal to $\Omega_i(s,\rho)$, and on
$\Pi_i(s,\rho)$ let $n_\rho(x):=\frac{x}{\rho}$.

\begin{lem}\label{lem:var-well-energy}
Assume $\cE_{\eps,a}(u;B_1)\leq M$ and
$\rho\in[\frac{9}{16},\frac{4}{5}]$.  With
\[
 \mathbb V(\rho):=\mathbb V_{\eps,a}(u;B_\rho),
 \quad
 O(\rho)=\int_{\partial B_\rho}e_{\eps,a}(u)\dd\cH^1,
\]
one has
\begin{equation}\label{eq:var-fixed-well}
 \cE_{\eps,a}\left(u;\Upsilon\left(\frac{\upmu_V}{4},\rho\right)\right)
 \leq C_M[\mathbb V(\rho)+\eps O(\rho)].
\end{equation}
If $\abs{u-\upsigma_1}<\kappa<\frac{\upmu_V}{4}$ on $\partial B_\rho$, then
\begin{equation}\label{eq:var-refined-well}
 \cE_{\eps,a}(u;\Upsilon(\kappa,\rho))
 \leq C_M[\kappa \mathbb V(\rho)+\eps O(\rho)].
\end{equation}
Consequently,
\begin{equation}\label{eq:var-full-by-P}
 \cE_{\eps,a}(u;B_\rho)
 \leq C_M[\mathbb V(\rho)+\eps O(\rho)].
\end{equation}
\end{lem}

\begin{proof}
Multiply \eqref{eq:var-pde} by $u-\upsigma_i$ on
$\Omega_i(s,\rho)$.  The signed identity is
\begin{align}
 Q_i(s,\rho)
 &:=\int_{\Omega_i(s,\rho)}
 \left[\eps\abs{\nabla u}^2
 +\frac{a}{\eps}\nabla V(u)\cdot(u-\upsigma_i)\right]\dd x
 \label{eq:var-Q}\\
 &=\eps s\int_{\Gamma_i(s,\rho)}
 \partial_{n_i}\abs{u-\upsigma_i}\dd\cH^1
 +\eps\int_{\Pi_i(s,\rho)}
 \partial_{n_\rho}u\cdot(u-\upsigma_i)\dd\cH^1.
 \label{eq:var-Qbdry}
\end{align}
No derivative of $a$ occurs.  By \eqref{eq:var-well},
\[
 \cE_{\eps,a}(u;\Omega_i(s,\rho))\leq C(V)Q_i(s,\rho),
\]
and on $\Pi_i(s,\rho)$,
\[
 \eps\abs{\partial_{n_\rho}u\cdot(u-\upsigma_i)}
 \leq C(V,a_-)\eps e_{\eps,a}(u).
\]

On the shell
\[
 \mathcal S
 :=\bigcup_i\left\{\frac{\upmu_V}{4}<\abs{u-\upsigma_i}
 <\frac{\upmu_V}{2}\right\}\cap B_\rho,
\]
Lemma~\ref{lem:var-far} yields
$\cE_{\eps,a}(u;\mathcal S)\leq C_M\mathbb V(\rho)$.  Coarea provides a common
regular $s_*\in[\frac{\upmu_V}{4},\frac{\upmu_V}{2}]$ such that
\[
 \eps\sum_i\int_{\Gamma_i(s_*,\rho)}\abs{\nabla u}\dd\cH^1
 \leq C_M\mathbb V(\rho).
\]
Insert this estimate in \eqref{eq:var-Qbdry}, sum over $i$, and use
$\Upsilon\left(\frac{\upmu_V}{4},\rho\right)
\subset\Upsilon(s_*,\rho)$ to prove
\eqref{eq:var-fixed-well}.

For regular $\kappa\in\left(0,\frac{\upmu_V}{4}\right)$, define
\begin{equation}\label{eq:target-annulus-definition}
 A_i(\kappa,s_*):=
 \{x\in B_\rho:\kappa<\abs{u(x)-\upsigma_i}<s_*\}.
\end{equation}
Testing
$-\Delta u+a\eps^{-2}\nabla V(u)=0$ on $A_i(\kappa,s_*)$ with
$\frac{u-\upsigma_i}{\abs{u-\upsigma_i}}$ yields
\begin{equation}\label{eq:target-annulus-flux}
\begin{aligned}
 &\int_{\Gamma_i(s_*,\rho)}
 \partial_{n_i}\abs{u-\upsigma_i}\,\dd\cH^1
 -\int_{\Gamma_i(\kappa,\rho)}
 \partial_{n_i}\abs{u-\upsigma_i}\,\dd\cH^1
 \\
 &=\int_{A_i(\kappa,s_*)}
 \frac{\abs{\nabla u}^2
 -\abs{\nabla\abs{u-\upsigma_i}}^2}{\abs{u-\upsigma_i}}\,\dd x
 \\
 &\quad+\int_{A_i(\kappa,s_*)}\frac{a}{\eps^2}\nabla V(u)\cdot
 \frac{u-\upsigma_i}{\abs{u-\upsigma_i}}\,\dd x
 \geq0.
\end{aligned}
\end{equation}
Here $A_i(\kappa,s_*)\cap\partial B_\rho=\varnothing$: for $i=1$ this
follows from $\abs{u-\upsigma_1}<\kappa$ on $\partial B_\rho$, while for
$i\ne1$ it follows from
$\abs{\upsigma_i-\upsigma_1}>2\upmu_V$.  Consequently,
\[
 \eps\sum_i\int_{\Gamma_i(\kappa,\rho)}
 \partial_{n_i}\abs{u-\upsigma_i}\dd\cH^1
 \leq C_M\mathbb V(\rho).
\]
The estimate for arbitrary $\kappa$ follows by taking regular values
$\kappa_m\to\kappa$ and using lower semicontinuity.
Insert this estimate in \eqref{eq:var-Qbdry} at $s=\kappa$ to obtain
\eqref{eq:var-refined-well}.  Finally,
$B_\rho=\Upsilon\left(\frac{\upmu_V}{4},\rho\right)\cup\Xi$; combine
\eqref{eq:var-fixed-well} and \eqref{eq:var-far} to prove
\eqref{eq:var-full-by-P}.
\end{proof}

\subsection{Pohozaev absorption and energy decay}

The variable-coefficient Pohozaev identity absorbs the coefficient error and
closes the planar energy-decay estimate.

\begin{lem}
\label{lem:var-pohozaev}
For almost every $0<r<1$,
\begin{equation}\label{eq:var-pohozaev}
\int_{B_r}(2a+x\cdot\nabla a)\frac{V(u)}{\eps}
=r\int_{\partial B_r}
\left\{\frac{\eps}{2}(\abs{\nabla_\tau u}^2-\abs{\partial_ru}^2)
+\frac{a}{\eps}V(u)\right\}.
\end{equation}
More generally, for almost every admissible radius of a ball $B_r(x_0)$ with
$\overline{B_r(x_0)}\subset B_1$,
\begin{equation}\label{eq:var-P-absorb}
\left(2-\frac{r\norm{\nabla a}_{L^\infty}}{a_-}\right)
\mathbb V_{\eps,a}(u;B_r(x_0))
\le r\int_{\partial B_r(x_0)}e_{\eps,a}(u).
\end{equation}
Thus, if $\frac{r\norm{\nabla a}_{L^\infty}}{a_-}\leq\frac{1}{2}$,
\begin{equation}\label{eq:var-P-boundary}
\mathbb V_{\eps,a}(u;B_r(x_0))
\le \frac{2r}{3}\int_{\partial B_r(x_0)}e_{\eps,a}(u).
\end{equation}
\end{lem}

\noindent\emph{Bethuel correspondence.}
For $a\equiv1$, this is Bethuel's disk Pohozaev identity
\cite[Lemma 3.8]{Bethuel2025}.  The term
$(x-x_0)\cdot\nabla a$ is the sole new contribution and is absorbed by the
potential term on the left of \eqref{eq:var-P-absorb}.

\begin{proof}
Set
\[
A_{ij}=e_{\eps,a}(u)\delta_{ij}
-\eps\partial_i u\cdot\partial_j u.
\]
The equation yields, in distributions,
\[
\partial_jA_{ij}=\frac{V(u)}{\eps}\partial_i a.
\]
Let $n_r(x):=\frac{x}{r}$ on $\partial B_r$.  Test with the radial field
$X(x):=x$, using radial cutoffs converging to $\mathbf1_{B_r}X$.
In dimension two,
$\operatorname{tr}A=\frac{2aV}{\eps}$, while on $\partial B_r$
\[
(A n_r)\cdot x
=r\left\{\frac{\eps}{2}(\abs{\nabla_\tau u}^2-\abs{\partial_ru}^2)
+\frac{a}{\eps}V(u)\right\}.
\]
This proves \eqref{eq:var-pohozaev}; Lipschitz regularity of $a$ is
enough by approximation or directly from the distributional identity.
Since
\[
\abs{\int_{B_r(x_0)}\frac{V}{\eps}(x-x_0)\cdot\nabla a}
\leq \frac{r\norm{\nabla a}_{L^\infty}}{a_-}\mathbb V_{\eps,a}(B_r(x_0)),
\]
and the absolute value of the boundary integrand is at most
$e_{\eps,a}$, \eqref{eq:var-P-absorb} follows.  Notice that the
coefficient error has been compared with the same potential term on the
left and absorbed.  No term $+\delta E$ remains.
\end{proof}

\begin{lem}\label{lem:var-planar-separation}
Let $w$ be $C^1$ on $\overline{B_\rho}$, let $s>0$ be a regular value,
and suppose $w<s$ on $\partial B_\rho$.  If
$\partial B_r\subset\{w>s\}$ for some $r<\rho$, then
$\{w=s\}$ contains a closed component with non-zero winding number about
the origin, lying outside $B_r$.  In particular its length is at least
$2\pi r$.
\end{lem}

\begin{proof}
The component of $\{w>s\}$ containing $\partial B_r$ is separated from
$\partial B_\rho$.  By regularity, its outer boundary is a finite union of
closed $C^1$ level curves.  At least one has non-zero winding number about
the origin; otherwise their union would not separate the two circles.
On $\R^2\backslash B_r$, angular projection to $\partial B_r$ is
$1$-Lipschitz, and a non-zero-winding loop covers that circle.  Its length
is therefore at least $2\pi r$.
\end{proof}

\begin{prop}[The $\frac{3}{2}$-decay]\label{prop:var-decay}
Under the assumptions of Theorem~\ref{thm:var-clearing-out},
\[
\cE_{\eps,a}\left(u;B_{\frac{1}{2}}\right)
\le C\{\cE_{\eps,a}(u;B_1)^{\frac{3}{2}}
+\eps \cE_{\eps,a}(u;B_1)\}.
\]
\end{prop}

\noindent\emph{Bethuel correspondence.}
This is the variable-coefficient version of Bethuel's
Proposition~1.12, whose proof occupies Section~5 of \cite{Bethuel2025}.
The smaller target ball used here causes no loss, and
Lemma~\ref{lem:var-pohozaev} controls the additional coefficient error.

\begin{proof}
Write $\mathcal E=\cE_{\eps,a}(u;B_1)$.  For
$\mathcal E\ge\eta_*$, with $\eta_*>0$ fixed, the estimate is
trivial after taking $C\ge\eta_*^{-\frac{1}{2}}$.  We hence assume
$\mathcal E$ small.

Choose $\rho\in[\frac{3}{4},\frac{4}{5}]$ with
\[
O(\rho)\le20\mathcal E.
\]
Lemma~\ref{lem:var-circle} yields a well $\upsigma$ such that
\[
\sup_{\partial B_\rho}\abs{u-\upsigma}
\le C_{\rm cir}\sqrt{20\mathcal E}.
\]
Set
\[
\kappa=4C_{\rm cir}\sqrt{20\mathcal E}.
\]
Decreasing $\eta_*$ ensures $\kappa<\frac{\upmu_V}{4}$ and
$\abs{u-\upsigma}\leq\frac{\kappa}{4}$ on $\partial B_\rho$.
Lemma~\ref{lem:var-well-energy} yields
\begin{equation}\label{eq:var-U-threehalf}
\cE_{\eps,a}(u;\Upsilon(\kappa,\rho))
\le C(\mathcal E^{\frac{3}{2}}+\eps\mathcal E).
\end{equation}

We next find a full inner circle in this well region.  Apply coarea to
$w=\chi_\upsigma\circ u$ for target values in
$[\frac{\kappa}{2},\kappa]$.  By \eqref{eq:var-MM}, some regular
$s_0\in[\frac{\kappa}{2},\kappa]$ satisfies
\[
\mathcal H^1(\{w=s_0\}\cap B_\rho)
\le C\frac{\mathcal E}{\kappa^2}.
\]
Indeed, on this target interval
$\abs{\nabla w}\leq\frac{\abs{\nabla(w^2)}}{\kappa}$; integrate by coarea and
divide by the interval length $\frac{\kappa}{2}$.
Enlarge the fixed numerical factor in the definition of $\kappa$, if
necessary, so that this length is at most $\frac{1}{32}$.  Since
$w<s_0$ on $\partial B_\rho$, no circle
$\partial B_r$, $r\in[\frac{5}{8},\rho]$, can satisfy $w>s_0$ everywhere:
otherwise a component of $\{w=s_0\}$ would surround $B_r$ and have
length at least $2\pi r>\frac{1}{32}$, by
Lemma~\ref{lem:var-planar-separation}.  If
$\sup_{\partial B_r}w>\kappa$, the circle therefore crosses
$\{w=s_0\}$.  Radial projection is $1$-Lipschitz, hence
\[
\abs{\{r\in[\frac{5}{8},\rho]:
\sup_{\partial B_r}w>\kappa\}}\leq\frac{1}{32}.
\]
Because $\rho\ge\frac{3}{4}$, the complementary set of good radii has length
at least $\frac{3}{32}$.  Fubini and \eqref{eq:var-U-threehalf} yield a good
$\tau\in[\frac{5}{8},\rho]$ such that
\begin{equation}\label{eq:var-inner-circle}
\int_{\partial B_\tau}e_{\eps,a}(u)
\le C(\mathcal E^{\frac{3}{2}}+\eps\mathcal E).
\end{equation}
The whole circle lies in $\{\abs{u-\upsigma}\leq\kappa\}$.
By Lemma~\ref{lem:var-pohozaev},
\begin{equation}\label{eq:var-improved-P}
\mathbb V_{\eps,a}\left(u;B_{\frac{5}{8}}\right)
\le \mathbb V_{\eps,a}(u;B_\tau)
\le C(\mathcal E^{\frac{3}{2}}+\eps\mathcal E).
\end{equation}

Choose $r_*\in[\frac{9}{16},\frac{5}{8}]$ with
\[
\int_{\partial B_{r_*}}e_{\eps,a}(u)
\le16\cE_{\eps,a}\left(u;B_{\frac{5}{8}}\right).
\]
Using \eqref{eq:var-full-by-P} and \eqref{eq:var-improved-P},
\[
\cE_{\eps,a}\left(u;B_{\frac{1}{2}}\right)
\le \cE_{\eps,a}(u;B_{r_*})
\le C\{\mathbb V_{\eps,a}(u;B_{r_*})
+\eps O(r_*)\}
\le C(\mathcal E^{\frac{3}{2}}+\eps\mathcal E).
\]
\end{proof}

\subsection{Corrected dyadic clearing-out}

We now iterate the decay estimate across dyadic scales and derive the
clearing-out consequences used for limiting measures.

For $u_R(y)=u(x_0+Ry)$, $a_R(y)=a(x_0+Ry)$, and
$\eps_R=\frac{\eps}{R}$,
\begin{equation}\label{eq:var-scale}
\cE_{\eps_R,a_R}(u_R;B_1)
=R^{-1}\cE_{\eps,a}(u;B_R(x_0)),\quad
\norm{\nabla a_R}_{L^\infty}
\leq R\norm{\nabla a}_{L^\infty}.
\end{equation}
Thus the coefficient class improves under blow-up.
Proposition~\ref{prop:var-decay} scales to
\begin{equation}\label{eq:var-scaled-decay}
\cE_{\eps,a}\left(u;B_{\frac{R}{2}}(x_0)\right)
\le C\left\{
\frac{\cE_{\eps,a}(u;B_R(x_0))^{\frac{3}{2}}}{\sqrt R}
+\frac{\eps}{R}\cE_{\eps,a}(u;B_R(x_0))\right\},
\end{equation}
provided $R\ge\eps$ and
$\frac{R\operatorname{Lip}(a)}{a_-}\le\frac{1}{2}$.

\Needspace{6\baselineskip}
\par\addvspace{\medskipamount}
\begin{lem}\label{lem:var-weak}
There is $\eta_{\rm w}=\eta_{\rm w}(V,a_-,a_+)>0$ such that, for
$0<\eps\le4$,
\[
\cE_{\eps,a}(u;B_1)\le\eta_{\rm w}\eps
\quad\Rightarrow\quad
u\left(B_{\frac{3}{4}}\right)\subset B\left(\upsigma,\frac{\upmu_V}{2}\right)
\]
for one $\upsigma\in\Sigma$.
\end{lem}

\begin{proof}
It is enough to choose $\eta_{\rm w}\leq1$.  The assumed smallness then
implies $\cE_{\eps,a}(u;B_1)\leq4$, so Lemma~\ref{lem:var-far} applies with
$M=4$.  By that lemma and the compactness of
$\{y:\abs y\leq C(4,V)\}$,
\begin{equation}\label{eq:weak-clearing-out-lipschitz}
 \norm{\nabla(V\circ u)}_{L^\infty(B_{\frac{4}{5}})}
 \leq\norm{DV(u)}_{L^\infty(B_{\frac{4}{5}})}
 \norm{\nabla u}_{L^\infty(B_{\frac{4}{5}})}
 \leq\frac{C_1}{\eps}.
\end{equation}
Set
\begin{equation}\label{eq:weak-clearing-out-v0}
 v_0:=\min\left\{V(y):\dist(y,\Sigma)\geq\frac{\upmu_V}{2},
 \ \abs y\leq C(4,V)\right\}>0.
\end{equation}
If $x_0\in B_{\frac{3}{4}}$ and
$\dist(u(x_0),\Sigma)\geq\frac{\upmu_V}{2}$, then
\eqref{eq:weak-clearing-out-lipschitz} implies
\begin{equation}\label{eq:weak-clearing-out-positive-ball}
 V(u(x))\geq\frac{v_0}{2}
 \quad\text{for }x\in B_{r_\eps}(x_0),
 \quad
 r_\eps:=\min\left\{\frac{v_0\eps}{2C_1},\frac{1}{40}\right\}.
\end{equation}
Consequently,
\begin{equation}\label{eq:weak-clearing-out-contradiction}
\begin{aligned}
 \frac{\eps}{a_-}\cE_{\eps,a}(u;B_1)
 &\geq\int_{B_1}V(u)\,\dd x
 \geq\frac{\pi v_0}{2}r_\eps^2.
\end{aligned}
\end{equation}
For $0<\eps\leq\eps_*:=\frac{C_1}{20v_0}$, the right-hand side is
$c_0\eps^2$; for $\eps\in[\eps_*,4]$, it is at least
$c_0\eps_*^2$.  Thus \eqref{eq:weak-clearing-out-contradiction} contradicts
$\cE_{\eps,a}(u;B_1)\leq\eta_{\rm w}\eps$ provided
\begin{equation}\label{eq:weak-clearing-out-eta-choice}
 \eta_{\rm w}<
 \min\left\{a_-c_0,\frac{a_-c_0\eps_*^2}{16}\right\}.
\end{equation}
Hence
\begin{equation}\label{eq:weak-clearing-out-union}
 u\left(B_{\frac{3}{4}}\right)\subset
 \bigcup_{\upsigma\in\Sigma}B\left(\upsigma,\frac{\upmu_V}{2}\right).
\end{equation}
The balls in \eqref{eq:weak-clearing-out-union} are pairwise disjoint and
$B_{\frac{3}{4}}$ is connected; continuity of $u$ therefore proves the assertion.
\end{proof}

\begin{lem}\label{lem:var-center}
Assume the hypotheses of Theorem~\ref{thm:var-clearing-out}, including that
$u$ is a weak solution of \eqref{eq:var-pde},
$\norm{\nabla a}_{L^\infty}\leq\frac{a_-}{2}$, and $0<\eps\leq1$.
There is $\eta_{\rm c}>0$, depending only on
$V,a_-,a_+$, such that
\[
\cE_{\eps,a}(u;B_1)\le\eta_{\rm c}
\quad\Rightarrow\quad
\dist(u(0),\Sigma)\le\frac{\upmu_V}{2}.
\]
\end{lem}

\noindent\emph{Bethuel correspondence.}
This is the perturbed form of Bethuel's dyadic clearing-out argument at the
centre \cite[Proposition 6.5]{Bethuel2025}.

\begin{proof}
Put $r_n=2^{-n}$ and $E_n=\cE_{\eps,a}(u;B_{r_n})$.
Fix $\eps_0>0$ sufficiently small for the logarithmic estimates below and
also $\eps_0\leq\eta_{\rm w}$.  For
$\eps\in[\eps_0,1]$, choose
$\eta_{\rm c}\leq\eta_{\rm w}\eps_0$ and apply weak clearing-out directly at
scale one.  We hence assume $\eps<\eps_0$.  If $E_0<\eps^2$, then
$E_0<\eta_{\rm w}\eps$ and weak clearing-out applies.  Otherwise, let $N$
be the last index such that
\[
r_N\ge\eps,\quad E_N\ge\frac{\eps^2}{r_N}.
\]
The index exists because $n=0$ qualifies.  If it qualifies at $N$, then
both inequalities qualify at every $n\leq N$, since
$E_n\geq E_N$ and
$\frac{\eps^2}{r_n}\leq\frac{\eps^2}{r_N}$.  Thus the recurrence below
can indeed be iterated through $N$.
Whenever these inequalities hold, \eqref{eq:var-scaled-decay} yields
\begin{equation}\label{eq:var-superlinear}
E_{n+1}\le2C\,2^{\frac{n}{2}}E_n^{\frac{3}{2}}.
\end{equation}
Indeed, the second term in \eqref{eq:var-scaled-decay} is then bounded
by the first.

With $A_n=-\log E_n$,
\[
A_{n+1}\ge\frac{3}{2}A_n-\frac{n}{2}\log2-\log(2C).
\]
Let
\[
\Gamma=\sum_{j=0}^\infty
\left(\frac{2}{3}\right)^{j+1}
\left(\frac{j}{2}\log2+\log(2C)\right)<\infty.
\]
Induction yields
\[
A_n\ge\left(\frac{3}{2}\right)^n(A_0-\Gamma).
\]
If $E_0\le e^{-(\Gamma+1)}$, then
\[
E_n\le\exp\left[-\left(\frac{3}{2}\right)^n\right].
\]
At the last admissible index,
\[
\eps^22^N\le E_N\le e^{-\left(\frac{3}{2}\right)^N},
\]
and consequently
\[
\left(\frac{3}{2}\right)^N+N\log2\leq2\abs{\log\eps},
\quad
N\leq\frac{\log(2\abs{\log\eps})}{\log\left(\frac{3}{2}\right)}.
\]
Writing $\gamma=\frac{\log2}{\log\left(\frac{3}{2}\right)}$, this implies
\[
r_N\geq(2\abs{\log\eps})^{-\gamma}.
\]
The correct comparison is
$\eps\ll\abs{\log\eps}^{-\gamma}$.  Hence, for all small $\eps$,
$r_N\ge2\eps$.  Maximality yields
\[
E_{N+1}<\frac{\eps^2}{r_{N+1}}.
\]
At scale $R=r_{N+1}$,
\[
\cE_{\frac{\eps}{R},a_R}(u_R;B_1)
=\frac{E_{N+1}}{R}<\left(\frac{\eps}{R}\right)^2,\quad
\frac{\eps}{R}\leq C\eps\abs{\log\eps}^\gamma.
\]
For small $\eps$, this is at most
$\eta_{\rm w}\left(\frac{\eps}{R}\right)$, and
Lemma~\ref{lem:var-weak} applies.
This proves the claim in the small-$\eps$ case and completes the proof.
\end{proof}

\begin{proof}[Completion of Theorem~\ref{thm:var-clearing-out}]
Proposition~\ref{prop:var-decay} proves \eqref{eq:var-decay}.
For each $x_0\in B_{\frac{3}{4}}$, rescale $B_{\frac{1}{4}}(x_0)$ to $B_1$.
Its normalized energy is at most $4\cE_{\eps,a}(u;B_1)$, and its
coefficient has Lipschitz seminorm at most $\frac{\delta_0}{4}$.
Choose $2\eta_0\le\frac{\eta_{\rm c}}{4}$.
Lemma~\ref{lem:var-center} yields pointwise confinement when $4\eps\le1$;
when $4\eps>1$, decrease $\eta_0$ and use weak clearing-out.
Continuity and separation of the well balls show that the same
$\upsigma$ works on all of $B_{\frac{3}{4}}$, proving
\eqref{eq:var-confine}.

Choose $r\in[\frac{5}{8},\frac{3}{4}]$ with
\[
\int_{\partial B_r}e_{\eps,a}(u)\le8\cE_{\eps,a}(u;B_1).
\]
Multiply \eqref{eq:var-pde} by $u-\upsigma$ on $B_r$:
\[
\int_{B_r}\left\{\eps\abs{\nabla u}^2+
\frac{a}{\eps}DV(u)\cdot(u-\upsigma)\right\}
=\eps\int_{\partial B_r}\partial_ru\cdot(u-\upsigma).
\]
By \eqref{eq:var-well} and confinement, the left side dominates
$c\cE_{\eps,a}(u;B_r)$.  Moreover,
\[
 \abs{u-\upsigma}\leq C\sqrt{V(u)},
 \quad \abs{\nabla u}\sqrt V\leq C e_{\eps,a}.
\]
Thus the boundary term, including its prefactor $\eps$, is at most
$C\eps \cE_{\eps,a}(u;B_1)$.  Since $B_{\frac{5}{8}}\subset B_r$,
\eqref{eq:var-small-inner} follows.
\end{proof}

\begin{lem}
\label{lem:var-uniform-energy-potential}
Assume in addition that $\norm{u}_{L^\infty(B_{\frac{4}{5}})}\leq L$.
There is $C=C(V,a_-,a_+,L)$ such that
\begin{equation}\label{eq:var-uniform-energy-potential}
 \cE_{\eps,a}\left(u;B_{\frac{1}{2}}\right)
 \leq C\left[
 \mathbb V_{\eps,a}\left(u;B_{\frac{3}{4}}\right)
 +\eps\cE_{\eps,a}\left(u;B_1\backslash B_{\frac{1}{2}}\right)
 \right].
\end{equation}
After scaling, suppose that $u$ is a weak solution on $B_r(x_0)$,
$a_-\leq a\leq a_+$ there,
$\norm{u}_{L^\infty(B_{\frac{4r}{5}}(x_0))}\leq L$, and $\eps\leq r$.
Then
\begin{equation}\label{eq:var-uniform-energy-potential-scaled}
 \cE_{\eps,a}\left(u;B_{\frac{r}{2}}(x_0)\right)
 \leq C\left[
 \begin{aligned}
 &
 \mathbb V_{\eps,a}\left(u;B_{\frac{3r}{4}}(x_0)\right)
 \\
 &+\frac{\eps}{r}\cE_{\eps,a}\left(
 u;B_r(x_0)\backslash B_{\frac{r}{2}}(x_0)
 \right)
 \end{aligned}
 \right].
\end{equation}
\end{lem}

\noindent\emph{Bethuel correspondence.}
For $a\equiv1$, estimate~\eqref{eq:var-uniform-energy-potential} is
Bethuel's Proposition~4.11, and
\eqref{eq:var-uniform-energy-potential-scaled} is its scaled
form~\cite[formula (4.40)]{Bethuel2025}.
The bounds $a_-\leq a\leq a_+$ only change the constant.

\begin{proof}

By the coarea identity, there exists
$\rho\in[\frac{9}{16},\frac{3}{4}]$ such that
\begin{equation}\label{eq:energy-potential-good-radius}
 O(\rho)
 \leq\frac{16}{3}
 \cE_{\eps,a}\left(u;B_{\frac{3}{4}}\backslash B_{\frac{9}{16}}\right)
 \leq\frac{16}{3}
 \cE_{\eps,a}\left(u;B_1\backslash B_{\frac{1}{2}}\right).
\end{equation}
The target-space multiplier argument proving
\eqref{eq:var-full-by-P} applies with a constant depending on
$V,a_-,a_+$ and $L$.  Indeed, the assumed $L^\infty$ bound and the
rescaled interior elliptic estimate provide
\[
 \eps\norm{\nabla u}_{L^\infty(B_{\frac{3}{4}})}\leq C(V,a_+,L).
\]
The remainder of that argument uses only \eqref{eq:var-well} and
$a_-\leq a\leq a_+$.  Hence
\begin{equation}\label{eq:energy-potential-radius-estimate}
 \cE_{\eps,a}(u;B_\rho)
 \leq C\left[
 \mathbb V_{\eps,a}(u;B_\rho)+\eps O(\rho)
 \right].
\end{equation}
Since $B_{\frac{1}{2}}\subset B_\rho\subset B_{\frac{3}{4}}$,
\eqref{eq:var-uniform-energy-potential} follows from
\eqref{eq:energy-potential-good-radius}--
\eqref{eq:energy-potential-radius-estimate}.

For the scaled statement, set
\[
 u_r(y):=u(x_0+ry),\quad
 a_r(y):=a(x_0+ry),\quad
 \eps_r:=\frac{\eps}{r}.
\]
Apply \eqref{eq:var-uniform-energy-potential} to
$(u_r,a_r,\eps_r)$ and use the scaling identities
\[
 \cE_{\eps_r,a_r}(u_r;U)
 =\frac{1}{r}\cE_{\eps,a}(u;x_0+rU),
 \quad
 \mathbb V_{\eps_r,a_r}(u_r;U)
 =\frac{1}{r}\mathbb V_{\eps,a}(u;x_0+rU).
\]
This proves \eqref{eq:var-uniform-energy-potential-scaled}.
\end{proof}

\begin{cor}
\label{cor:var-measure-package}
Fix $x_0\in\R^2$ and $R>0$.  Let $a_j\in W^{1,\infty}(B_R(x_0))$ satisfy
\[
 a_-\leq a_j\leq a_+,
 \quad
 \frac{R\norm{\nabla a_j}_{L^\infty(B_R(x_0))}}{a_-}\leq\frac{1}{2}.
\]
Let $\frac{\eps_j}{R}\to0$, and let $u_j\in H^1(B_R(x_0);\R^k)$ be weak
solutions of
\[
 -\eps_j\Delta u_j+\frac{a_j}{\eps_j}\nabla V(u_j)=0
 \quad\text{in }B_R(x_0),
\]
with $\norm{u_j}_{L^\infty(B_R(x_0))}\leq L$.  Suppose
\[
 e_{\eps_j,a_j}(u_j)\dd x\stackrel{*}{\rightharpoonup}\nustar,
 \quad
 \frac{a_jV(u_j)}{\eps_j}\dd x
 \stackrel{*}{\rightharpoonup}\zetastar
\]
on $B_R(x_0)$.  After reducing $\eta_0$ by a universal numerical factor,
the following assertions hold.
\begin{enumerate}[label=$(\theenumi)$]
\item If
$\nustar(\overline{B_r(x)})<\eta_0r$ and
$\overline{B_r(x)}\subset B_R(x_0)$, then
$\nustar\left(B_{\frac{r}{2}}(x)\right)=0$.
\item If $x\in\supp\nustar$ and
$\overline{B_r(x)}\subset B_R(x_0)$, then
\[
 \nustar(B_r(x))\geq\eta_0r.
\]
\item For every ball $B_r(x)$ with
$\overline{B_r(x)}\subset B_R(x_0)$,
\begin{equation}\label{eq:var-nu-zeta-comparison}
 \nustar\left(B_{\frac{r}{2}}(x)\right)
 \leq C(V,a_-,a_+,L)\zetastar\left(B_{\frac{3r}{4}}(x)\right).
\end{equation}
\end{enumerate}
\end{cor}

\begin{proof}
For (i), choose $r'\in(\frac{4r}{5},r)$ with
$\nustar(\partial B_{r'}(x))=0$ and
$\frac{\nustar(B_{r'}(x))}{r'}<2\eta_0$.  Weak convergence yields the same strict
normalized bound for large $j$.  Scale $B_{r'}(x)$ to $B_1$
and use \eqref{eq:var-small-inner}; its right side is
$C\frac{\eps_j}{r'}\cE_{\eps_j,a_j}(B_{r'})\to0$.  Since
$B_{\frac{r}{2}}\subset B_{\frac{5r'}{8}}$, this proves
$\nustar\left(B_{\frac{r}{2}}\right)=0$.
For (ii), argue by contradiction.  If
$\nustar(B_r(x))<\eta_0r$, choose
\[
 s\in\left(\frac{\nustar(B_r(x))}{\eta_0},r\right).
\]
Then
$\nustar(\overline{B_s(x)})\leq\nustar(B_r(x))<\eta_0s$, so (i) implies
$\nustar\left(B_{\frac{s}{2}}(x)\right)=0$, contradicting
$x\in\supp\nustar$.

For (iii), first take $r$ such that the circles of radii
$\frac{r}{2}$, $\frac{3r}{4}$, and $r$ carry no limiting mass.  Applying
\eqref{eq:var-uniform-energy-potential-scaled} yields
\begin{equation}\label{eq:var-nu-zeta-prelimit-comparison}
 \cE_{\eps_j,a_j}\left(u_j;B_{\frac{r}{2}}(x)\right)
 \leq C\left[
 \begin{aligned}
 &
 \mathbb V_{\eps_j,a_j}\left(u_j;B_{\frac{3r}{4}}(x)\right)
 \\
 &+\frac{\eps_j}{r}
 \cE_{\eps_j,a_j}\left(
 u_j;B_r(x)\backslash B_{\frac{r}{2}}(x)
 \right)
 \end{aligned}
 \right].
\end{equation}
Weak-star convergence on the compact set $\overline{B_r(x)}$ implies
\[
 \sup_j\cE_{\eps_j,a_j}(u_j;B_r(x))<+\infty.
\]
Since $\frac{\eps_j}{r}\to0$, the second term on the right-hand side of
\eqref{eq:var-nu-zeta-prelimit-comparison} tends to zero.  Passing to the
limit therefore yields \eqref{eq:var-nu-zeta-comparison} for such radii;
monotone approximation proves it for every admissible $r$.
\end{proof}

\begin{rem}
If the even-reflected conformal coefficient has
$\operatorname{Lip}(a)\le K$, then on a radius-$R$ disk the
rescaled seminorm is at most $RK$.  The package applies uniformly for
\[
R\le R_0:=\min\left\{R_{\rm chart},\frac{a_-}{2K}\right\}.
\]
The scale-invariant clearing-out hypothesis is
$R^{-1}\cE_{\eps,a}(u;B_R)\le2\eta_0$.
\end{rem}
 % Completion of the general curved-boundary argument.

\section{Boundary stress and density estimates}
\label{sec:boundary-analysis}
\label{sec:curved-completion}

We now analyze the part of the limiting measures carried by the boundary.
The argument is performed before boundary absolute continuity is known.

\subsection{Boundary stress before absolute continuity}

The next observation does not use boundary densities.

\begin{lem}
\label{lem:boundary-stress-measure}
Let $\Gamma\subset\partial\Omega$ be a $C^2$ arc and let
\[
 T_j=\upnu_j I_2-(\upmu_{j,ab})_{a,b=1}^2
 \stackrel{*}{\rightharpoonup}T
\]
be the Allen--Cahn stress measures in original coordinates.  If $\tau$
and $n$ are the unit tangent and inward unit normal along $\Gamma$, then,
as Radon measures on $\Gamma$,
\begin{equation}\label{eq:boundary-stress-measure}
 T_{nn}\llcorner\Gamma=0,
 \quad T_{\tau n}\llcorner\Gamma=0,
 \quad T\llcorner\Gamma
 =2(\tau\otimes\tau)(\zetastar\llcorner\Gamma).
\end{equation}
Moreover, $\zetastar\llcorner\Gamma$ has no atoms.
Equivalently,
\begin{equation}\label{eq:boundary-discrepancy-measure}
 \begin{aligned}
  \upmu_{\star,n\tau}\llcorner\Gamma&=0,\\
  (\upmu_{\star,nn}-\upmu_{\star,\tau\tau})\llcorner\Gamma
  &=2\zetastar\llcorner\Gamma,\\
  \nustar\llcorner\Gamma&=\upmu_{\star,nn}\llcorner\Gamma.
\end{aligned}
\end{equation}
No absolute-continuity hypothesis is used.
\end{lem}

\begin{proof}
For every $X\in C_c^1(\Omega\cup\Gamma;\R^2)$ tangent to $\Gamma$, the
Neumann condition and \eqref{eq:inner-variation} yield
\begin{equation}\label{eq:stress-limit-tangent}
 \int_{\overline\Omega}T:DX=0.
\end{equation}
Use tubular coordinates $(s,t)$, where $t\geq0$ is inward distance.  Let
$\chi\in C_c^1([0,2))$ equal one near zero.  For
$\phi\in C_c^1(\Gamma)$ set
\[
 X_\delta^{(n)}=t\chi\left(\frac{t}{\delta}\right)\phi(s)n(s),
 \quad
 X_\delta^{(\tau)}=t\chi\left(\frac{t}{\delta}\right)\phi(s)\tau(s).
\]
Both fields vanish on $\Gamma$ and are admissible.  Their derivatives are
bounded independently of $\delta$, their supports lie in
$\{0\leq t<2\delta\}$, and on $\Gamma$
\[
 DX_\delta^{(n)}=\phi\,n\otimes n,
 \quad
 DX_\delta^{(\tau)}=\phi\,\tau\otimes n.
\]
Writing $T^{\partial}:=T\llcorner\Gamma$ and
$T^\circ:=T\llcorner(\Omega\cap\supp X_\delta)$, we have
\begin{align}
 0
 &=\int_\Gamma\phi\,\dd T_{nn}^{\partial}
 +\int_{\{0<t<2\delta\}}T^\circ:DX_\delta^{(n)},
 \label{eq:boundary-normal-test}\\
 0
 &=\int_\Gamma\phi\,\dd T_{\tau n}^{\partial}
 +\int_{\{0<t<2\delta\}}T^\circ:DX_\delta^{(\tau)}.
 \label{eq:boundary-mixed-test}
\end{align}
Since
\begin{equation}\label{eq:boundary-collar-vanishing}
 \sup_{\delta>0}\norm{DX_\delta^{(n)}}_\infty
 +\sup_{\delta>0}\norm{DX_\delta^{(\tau)}}_\infty<+\infty,
 \quad
 \abs T(\{0<t<2\delta\})\to0,
\end{equation}
letting $\delta\downarrow0$ in
\eqref{eq:boundary-normal-test}--\eqref{eq:boundary-mixed-test} implies
\begin{equation}\label{eq:boundary-two-stress-components}
 T_{nn}\llcorner\Gamma=T_{\tau n}\llcorner\Gamma=0.
\end{equation}
The exact diffuse identity
\begin{equation}\label{eq:diffuse-trace-identity}
 \upnu_j=\frac{1}{2}\tr\upmu_j+\upzeta_j,
 \quad
 \tr T_j=2\upzeta_j,
\end{equation}
passes to the limit, so $\tr T=2\zetastar$.  Symmetry of $T$ and
\eqref{eq:boundary-two-stress-components} yield
\begin{equation}\label{eq:boundary-stress-matrix}
 T\llcorner\Gamma
 =\begin{pmatrix}2\zetastar&0\\0&0\end{pmatrix}_{(\tau,n)}
 =2(\tau\otimes\tau)(\zetastar\llcorner\Gamma),
\end{equation}
which proves
\eqref{eq:boundary-stress-measure} and
\eqref{eq:boundary-discrepancy-measure}.

Fix $p\in\Gamma$, normalize the tubular coordinate by $s(p)=0$, and choose
$\chi\in C_c^1([0,2))$, $\chi=1$ on $[0,1]$.  Define
\begin{equation}\label{eq:atom-test-field}
 Y_\delta(s,t)
 :=\chi\left(\frac{(s^2+t^2)^{\frac{1}{2}}}{\delta}\right)s\tau(s).
\end{equation}
Then
\begin{equation}\label{eq:atom-test-properties}
 Y_\delta\cdot\nOmega=0\text{ on }\Gamma,
 \quad Y_\delta(p)=0,
 \quad DY_\delta(p)=\tau(p)\otimes\tau(p),
 \quad\sup_\delta\norm{DY_\delta}_\infty<+\infty.
\end{equation}
By \eqref{eq:boundary-stress-matrix},
\begin{equation}\label{eq:no-boundary-atom}
\begin{aligned}
 0=\lim_{\delta\downarrow0}\int T:DY_\delta
 &=T(\{p\}):DY_\delta(p)
 =2\zetastar(\{p\}).
\end{aligned}
\end{equation}
\end{proof}

\subsection{Conformal doubling and the forced radial identity}

Fix $p\in\partial\Omega$ and use the conformal chart from
Section~2.4, whose image agrees with $\Omega$ near $p$.  Write
\[
 v_j=u_j\circ\Psi,
 \quad a=\abs{\Psi'}^2.
\]
The equation and energy become
\begin{equation}\label{eq:curved-conformal-pde}
 -\eps_j\Delta v_j+\frac{a}{\eps_j}\nabla V(v_j)=0,
 \quad
 \partial_2v_j=0\quad\text{on }\{x_2=0\},
\end{equation}
and
\[
 \cE_{\eps_j,a}(v_j;U)
 =\int_U\left[
 \frac{\eps_j}{2}\abs{\nabla v_j}^2
 +\frac{a}{\eps_j}V(v_j)
 \right]\dd x.
\]
Evenly reflect $v_j$ and $a$ to $B_{2\rho}$.  Then
\begin{equation}\label{eq:even-coefficient}
 0<a_-\leq a\leq a_+,
 \quad a\in W^{1,\infty}(B_{2\rho}),
 \quad a(x_1,-x_2)=a(x_1,x_2).
\end{equation}
The reflected map is a weak solution of
\eqref{eq:curved-conformal-pde} on the full ball.

We shall repeatedly use the exact boundary factor in the doubled limits.
If $\lambda_j$ is a scalar diffuse measure in the upper half and
$\widetilde\lambda_j=\lambda_j+R_\#\lambda_j$, write
$\lambda=\lambda^\circ+\lambda^\partial$ for the decomposition of its
weak limit into the open-half and fixed-line parts.  Then
\begin{equation}\label{eq:curved-double-measure}
 \widetilde\lambda
 =\lambda^\circ+R_\#\lambda^\circ+2\lambda^\partial.
\end{equation}
For tensor measures, the open-half reflection is
\[
 Q(R_\#\Lambda^\circ)Q,
 \quad
 \Lambda^\circ=(\lambda_{ab}^\circ)_{a,b=1}^2,
 \quad
 Q:=\operatorname{diag}(1,-1).
\]
The diagonal fixed-line entries again acquire the factor two.  This follows
from Lemma~\ref{lem:measure-double}.  In particular, the restriction of a
doubled boundary measure must not be identified with its one-sided limit.

Set
\begin{equation}\label{eq:conformal-diffuse-measures}
\begin{aligned}
 \upnu_j^a
 :=\left[\frac{\eps_j}{2}\abs{\nabla v_j}^2
 +\frac{a}{\eps_j}V(v_j)\right]\dd x,
 \quad
 \upzeta_j^a:=\frac{a}{\eps_j}V(v_j)\dd x,
 \quad
 \upmu_{j,ab}:=\eps_j\partial_av_j\cdot\partial_bv_j\dd x,\\
 h_j:=\frac{V(v_j)}{\eps_j}\nabla a\dd x.
\end{aligned}
\end{equation}
Here $h_j$ is an $\R^2$-valued Radon measure and
$\upmu_j:=(\upmu_{j,ab})_{a,b=1}^2$ is a symmetric matrix-valued measure.
After passing to a further subsequence,
\begin{equation}\label{eq:conformal-measure-limits}
 \upnu_j^a\stackrel{*}{\rightharpoonup}\nustar^a,
 \quad
 \upzeta_j^a\stackrel{*}{\rightharpoonup}\zetastar^a,
 \quad
 \upmu_{j,ab}\stackrel{*}{\rightharpoonup}\upmu_{\star,ab}.
\end{equation}
Moreover, $h_j\stackrel{*}{\rightharpoonup}h$, and
\begin{equation}\label{eq:h-domination}
 \abs{h}\leq B\zetastar^a,
 \quad B:=\frac{\norm{\nabla a}_{L^\infty}}{a_-}.
\end{equation}
Indeed,
\begin{equation}\label{eq:h-domination-diffuse}
 \abs{h_j}
 =\frac{V(v_j)}{\eps_j}\abs{\nabla a}\,\dd x
 \leq\frac{\norm{\nabla a}_{L^\infty}}{a_-}\upzeta_j^a,
\end{equation}
and \eqref{eq:h-domination} follows by weak-star lower semicontinuity.
We do not identify $h$ with $\frac{\nabla a}{a}\zetastar^a$ on the fixed line,
because the even coefficient has two one-sided normal derivatives.

For a centre $p$ on the fixed line, translate $p$ to the origin and define
on $B_\rho\backslash\{0\}$
\begin{equation}\label{eq:polar-frame-definition}
 e_r(x):=\frac{x}{\abs x},
 \quad
 e_\theta(x):=(-e_r^2(x),e_r^1(x)).
\end{equation}
For $v\in W^{1,2}(B_\rho;\R^k)$ and
$\alpha,\beta\in\{r,\theta\}$, set
\begin{equation}\label{eq:polar-derivative-definition}
 \partial_rv:=e_r\cdot\nabla v,
 \quad
 \nabla_\theta v:=e_\theta\cdot\nabla v,
 \quad
 \upmu_{\star,\alpha\beta}
 :=(e_\alpha\otimes e_\beta):\upmu_\star.
\end{equation}
Define the signed Radon measure
\begin{equation}\label{eq:radial-defect}
 \mathcal N:=2\zetastar^a+\upmu_{\star,rr}-\upmu_{\star,\theta\theta},
\end{equation}
and extend it by zero at the centre.

\begin{lem}\label{lem:radial-defect-positive}
The radial defect measure $\mathcal N$ defined in
\eqref{eq:radial-defect} is non-negative.
\end{lem}

\begin{proof}
We decompose the doubled ball into
\[
 B_\rho
 =
 B_\rho^+\cup B_\rho^-
 \cup(\Gamma_\rho\backslash\{0\})\cup\{0\}
\]
and determine the restriction of $\mathcal N$ on each part.

\medskip
\noindent
\emph{Interior contribution.}
On $B_\rho^+\cup B_\rho^-$, Bethuel's interior discrepancy relations,
transported through the conformal chart, yield
\begin{equation}\label{eq:defect-interior-matrix}
 \begin{aligned}
  \upmu_{\star,n\tau}&=0,
  &\upmu_{\star,nn}&=e\,\cH^1\llcorner\Sstar,\\
  \upmu_{\star,\tau\tau}
  &=(e-2\Uptheta)\,\cH^1\llcorner\Sstar,
  &\zetastar^a&=\Uptheta\,\cH^1\llcorner\Sstar.
\end{aligned}
\end{equation}
At $\cH^1$-almost every point of $\Sstar\backslash\Gamma_\rho$, write
\[
 e_r=\cos\gamma\,\tau+\sin\gamma\,n,
 \quad
 e_\theta=-\sin\gamma\,\tau+\cos\gamma\,n.
\]
Since $\upmu_{\star,n\tau}=0$, one has
\begin{equation}\label{eq:defect-polar-components}
\begin{aligned}
 \upmu_{\star,rr}
 &=
 \cos^2\gamma\,\upmu_{\star,\tau\tau}
 +\sin^2\gamma\,\upmu_{\star,nn},
 \\
 \upmu_{\star,\theta\theta}
 &=
 \sin^2\gamma\,\upmu_{\star,\tau\tau}
 +\cos^2\gamma\,\upmu_{\star,nn}.
\end{aligned}
\end{equation}
Consequently,
\[
\begin{aligned}
 \upmu_{\star,rr}-\upmu_{\star,\theta\theta}
 &=
 \cos(2\gamma)
 \left(
 \upmu_{\star,\tau\tau}-\upmu_{\star,nn}
 \right)
 \\
 &=
 -2\Uptheta\cos(2\gamma)\,
 \cH^1\llcorner\Sstar.
\end{aligned}
\]
It follows that
\begin{equation}\label{eq:defect-interior-positive}
 \mathcal N\llcorner(B_\rho^+\cup B_\rho^-)
 =
 4\Uptheta\sin^2\gamma\,
 \cH^1\llcorner
 \left(\Sstar\cap(B_\rho^+\cup B_\rho^-)\right)
 \geq0.
\end{equation}

\medskip
\noindent
\emph{Fixed-line contribution.}
Let $x\in\Gamma_\rho\backslash\{0\}$.  Since the radial centre is the
origin,
\[
 e_r(x)=\pm\tau_\Gamma,
 \quad
 e_\theta(x)=\pm n_\Gamma.
\]
The boundary stress identity, established before doubling, implies
\[
 T\llcorner\Gamma_\rho
 =
 2(\tau_\Gamma\otimes\tau_\Gamma)
 \left(\zetastar^a\llcorner\Gamma_\rho\right).
\]
Equivalently,
\[
 \upmu_{\star,\tau\tau}
 -\upmu_{\star,nn}
 =
 -2\zetastar^a
 \quad\text{on }\Gamma_\rho.
\]
Hence
\begin{equation}\label{eq:defect-boundary-zero}
\begin{aligned}
 \mathcal N\llcorner(\Gamma_\rho\backslash\{0\})
 &=
 \left(
 2\zetastar^a
 +\upmu_{\star,rr}
 -\upmu_{\star,\theta\theta}
 \right)
 \llcorner(\Gamma_\rho\backslash\{0\})
 \\
 &=
 \left(
 2\zetastar^a
 +\upmu_{\star,\tau\tau}
 -\upmu_{\star,nn}
 \right)
 \llcorner(\Gamma_\rho\backslash\{0\})
 =0.
\end{aligned}
\end{equation}

\medskip
\noindent
\emph{Contribution at the centre.}
The no-atom argument yields
\[
 \zetastar^a(\{0\})=0,
 \quad
 T(\{0\})=0.
\]
Since $T=\nustar I_2-\upmu_\star$, it follows that
\[
 \upmu_\star(\{0\})=\nustar(\{0\})I_2.
\]
Thus the atomic part of $\upmu_\star$ at the origin is isotropic, and
therefore
\[
 \left(
 \upmu_{\star,rr}-\upmu_{\star,\theta\theta}
 \right)(\{0\})=0.
\]
Together with $\zetastar^a(\{0\})=0$, this yields
\begin{equation}\label{eq:defect-centre-zero}
 \mathcal N(\{0\})=0.
\end{equation}

Combining \eqref{eq:defect-interior-positive},
\eqref{eq:defect-boundary-zero}, and
\eqref{eq:defect-centre-zero}, we conclude that
\[
 \mathcal N\geq0
 \quad\text{on }B_\rho.
\]
\end{proof}

\begin{rem}
In Bethuel's interior setting, the discrepancy relations directly yield
\[
 \mathcal N=4\Uptheta\sin^2\gamma\,\cH^1\llcorner\Sstar\geq0.
\]
After conformal doubling, the same computation applies away from the fixed
line, while the new boundary stress identity and the no-atom argument are
needed to show that the fixed-line and centre contributions vanish.
\end{rem}

\begin{lem}\label{lem:forced-monotonicity}
Let
\[
 Z(r):=\zetastar^a(B_r),
 \quad F(r):=\frac{Z(r)}{r},
 \quad H(r):=\int_{B_r}x\cdot\dd h(x).
\]
For almost every $0<r_0<r_1<\rho$,
\begin{equation}\label{eq:forced-radial-limit}
 F(r_1)-F(r_0)
 =\int_{B_{r_1}\backslash B_{r_0}}
 \frac{1}{4\abs{x}}\dd\mathcal N(x)
 +\int_{r_0}^{r_1}\frac{H(s)}{2s^2}\dd s.
\end{equation}
Define the right-continuous representative
\[
 F^+(r):=\frac{\zetastar^a(\overline B_r)}{r}.
\]
Consequently,
\begin{equation}\label{eq:gronwall-monotonicity}
 r\mapsto e^{\frac{Br}{2}}F^+(r)\quad\text{is non-decreasing},
\end{equation}
and
\begin{equation}\label{eq:linear-zeta-bound}
 \frac{\zetastar^a(B_r)}{r}
 \leq F^+(r)
 \leq e^{\frac{B(R-r)}{2}}F^+(R)
 =e^{\frac{B(R-r)}{2}}\frac{\zetastar^a(\overline B_R)}{R}
 \quad(0<r<R<\rho).
\end{equation}
\end{lem}

\begin{rem}
When $a\equiv1$, one has $h=0$ and $B=0$, so
\eqref{eq:forced-radial-limit} reduces to
\cite[Lemma 1.24]{Bethuel2025} and
\eqref{eq:gronwall-monotonicity} becomes the exact monotonicity of
\cite[Proposition 1.23]{Bethuel2025}.  For variable $a$, the force generated by
$\nabla a$ is controlled by $\abs h\leq B\zetastar^a$, yielding the
exponentially corrected almost monotonicity above.
\end{rem}

\begin{proof}
The coefficient stress $A_j:=\upnu_j^aI-\upmu_j$ satisfies
$\diver A_j=h_j$.  Define
\begin{equation}\label{eq:diffuse-radial-quantities}
 Z_j(r):=\upzeta_j^a(B_r),
 \quad F_j(r):=\frac{Z_j(r)}{r},
 \quad H_j(r):=\int_{B_r}x\cdot\dd h_j(x).
\end{equation}
Testing $\diver A_j=h_j$ with radial cutoffs converging to
$x\mathbf1_{B_r}$ yields, for almost every $r$,
\begin{equation}\label{eq:variable-pohozaev-limit-proof}
 2Z_j(r)+H_j(r)
 =r\int_{\partial B_r}\left[
 \frac{\eps_j}{2}
 (\abs{\nabla_\theta v_j}^2-\abs{\partial_rv_j}^2)
 +\frac{a}{\eps_j}V(v_j)
 \right]\dd\cH^1.
\end{equation}
Since
$Z_j'(r)=\int_{\partial B_r}\frac{aV(v_j)}{\eps_j}\,\dd\cH^1$ for almost every
$r$, \eqref{eq:variable-pohozaev-limit-proof} yields
\begin{equation}\label{eq:diffuse-F-derivative}
 F_j'(r)=\frac{D_j(r)}{4r}+\frac{H_j(r)}{2r^2},
\end{equation}
where
\begin{equation}\label{eq:diffuse-defect-slice}
 D_j(r):=\int_{\partial B_r}\left[
 2\frac{a}{\eps_j}V(v_j)
 +\eps_j(\abs{\partial_rv_j}^2-\abs{\nabla_\theta v_j}^2)
 \right]\dd\cH^1.
\end{equation}
Equivalently, with
\begin{equation}\label{eq:diffuse-defect-measure}
 \mathcal N_j
 :=\left[
 2\frac{a}{\eps_j}V(v_j)
 +\eps_j(\abs{\partial_rv_j}^2-\abs{\nabla_\theta v_j}^2)
 \right]\dd x,
\end{equation}
integration of \eqref{eq:diffuse-F-derivative} yields
\begin{equation}\label{eq:diffuse-radial-integrated}
 F_j(r_1)-F_j(r_0)
 =\int_{B_{r_1}\backslash B_{r_0}}
 \frac{1}{4\abs x}\,\dd\mathcal N_j(x)
 +\int_{r_0}^{r_1}\frac{H_j(s)}{2s^2}\,\dd s.
\end{equation}
Choose $r_0,r_1$ so that every limiting measure in
\eqref{eq:conformal-measure-limits} assigns zero mass to
$\partial B_{r_0}\cup\partial B_{r_1}$.  Then
\begin{align}
 \mathcal N_j&\stackrel{*}{\rightharpoonup}\mathcal N
 &&\text{on }B_{r_1}\backslash\overline B_{r_0},
 \label{eq:defect-weak-limit}\\
 \int_{r_0}^{r_1}\frac{H_j(s)}{2s^2}\,\dd s
 &=\int_{B_{r_1}}K_{r_0,r_1}(x)\cdot\dd h_j(x)
 \to
 \int_{B_{r_1}}K_{r_0,r_1}(x)\cdot\dd h(x),
 \label{eq:forcing-Fubini-limit}
\end{align}
where
\begin{equation}\label{eq:forcing-kernel}
 K_{r_0,r_1}(x)
 :=\frac{x}{2}\int_{\max\{r_0,\abs x\}}^{r_1}s^{-2}\,\dd s
 \ \mathbf1_{B_{r_1}}(x).
\end{equation}
Passing to the limit in \eqref{eq:diffuse-radial-integrated} proves
\eqref{eq:forced-radial-limit} at such radii; approximation from above and
below proves it for almost every $r_0,r_1$.

By \eqref{eq:h-domination},
\begin{equation}\label{eq:forcing-H-bound}
 \abs{H(s)}
 \leq s\abs h(B_s)
 \leq Bs\zetastar^a(B_s)=Bs^2F(s).
\end{equation}
Since $\mathcal N\geq0$, equations
\eqref{eq:forced-radial-limit} and \eqref{eq:forcing-H-bound} imply, as
Radon measures on $(0,\rho)$,
\begin{equation}\label{eq:F-BV-inequality}
 \dd F+\frac{B}{2}F(r)\,\dd r\geq0.
\end{equation}
For the right-continuous representative,
\begin{equation}\label{eq:integrating-factor-measure}
 \dd\left(e^{\frac{Br}{2}}F^+(r)\right)
 =e^{\frac{Br}{2}}\left(\dd F+\frac{B}{2}F(r)\,\dd r\right)\geq0.
\end{equation}
This is \eqref{eq:gronwall-monotonicity}; evaluating
\eqref{eq:integrating-factor-measure} on $(r,R]$ yields
\eqref{eq:linear-zeta-bound}.
\end{proof}

\subsection{Boundary density and absolute continuity}

The forced radial identity used in this subsection is proved in
Lemma~\ref{lem:forced-monotonicity} above.

Let $\Sstar:=\supp\nustar$ in $\overline\Omega$.  The clearing-out package
Corollary~\ref{cor:var-measure-package}, transferred through the finitely many
bi-Lipschitz boundary charts, applies to balls with arbitrary centres in
the corresponding doubled charts, not only to centres on the fixed line.
Together with a finite cover of the remaining compact interior, the
smaller charts have a positive Lebesgue number.  Hence there are uniform
constants $r_c,\eta_c,C_c>0$ such that
\begin{align}
p\in \Sstar,\quad0<r<r_c
 &\Rightarrow
 \nustar(B_r(p)\cap\overline\Omega)\geq\eta_c r,
 \label{eq:CO-interface}\\
\nustar\left(B_{\frac{r}{2}}(p)\cap\overline\Omega\right)
&\leq C_c\zetastar(B_r(p)\cap\overline\Omega).
\label{eq:CP-interface}
\end{align}
The constants can be chosen uniformly because $\partial\Omega$ is compact.
In flat coordinates write $p=(p_1,d)$, where $d\geq0$.  If
$y=(y_1,-q)$ lies below the fixed line, then
\begin{equation}\label{eq:crossing-ball-reflection-case}
 \abs{Ry-p}^2
 =(y_1-p_1)^2+(q-d)^2
 \leq(y_1-p_1)^2+(q+d)^2
 =\abs{y-p}^2.
\end{equation}
Consequently, for every non-negative diffuse measure $\lambda_j$ and its
even double $\widetilde\lambda_j$,
\begin{equation}\label{eq:doubled-physical-ball-comparison}
 \widetilde\lambda_j(B_r(p))
 \leq2\lambda_j(B_r(p)\cap\overline{B_{2\rho}^+}).
\end{equation}
Equation \eqref{eq:doubled-physical-ball-comparison}, followed by the fixed
bi-Lipschitz distortion of $\Psi$, yield
\eqref{eq:CO-interface}--\eqref{eq:CP-interface} with uniform constants.

\begin{prop}[Linear boundary densities]
\label{prop:boundary-linear-density}
For every compact boundary arc $\Gamma'\subset\subset\partial\Omega$ there are
$r_*,c_*,C_*>0$ such that, for $p\in \Sstar\cap\Gamma'$ and $0<r<r_*$,
\begin{equation}\label{eq:two-sided-boundary-density}
 c_*r\leq\zetastar(B_r(p)\cap\overline\Omega)
 \leq\nustar(B_r(p)\cap\overline\Omega)\leq C_*r.
\end{equation}
For each $p\in\Gamma'$ the upper bounds remain valid, with the lower
bound omitted.
\end{prop}

\begin{proof}
Choose $R\in(0,\rho)$ such that $B_R(p)$ remains in one doubled chart for
every $p$ in a fixed smaller boundary arc.  By
\eqref{eq:linear-zeta-bound},
\begin{equation}\label{eq:boundary-zeta-upper-chart}
 \zetastar^a(B_r(p))
 \leq e^{\frac{BR}{2}}\frac{r}{R}\zetastar^a(\overline B_R(p))
 \leq\frac{2e^{\frac{BR}{2}}M_0}{R}r.
\end{equation}
Using \eqref{eq:curved-double-measure} and the bi-Lipschitz bounds of
$\Psi$, shrink $r_*$ so that
\begin{equation}\label{eq:boundary-zeta-upper-physical}
 \zetastar(B_r(p)\cap\overline\Omega)\leq C_1r
 \quad(p\in\Gamma',\ 0<r<2r_*).
\end{equation}
Then \eqref{eq:CP-interface} at radius $2r$ yields
\begin{equation}\label{eq:boundary-nu-upper}
 \nustar(B_r(p)\cap\overline\Omega)
 \leq C_c\zetastar(B_{2r}(p)\cap\overline\Omega)
 \leq2C_cC_1r.
\end{equation}
If $p\in\Sstar$, equations \eqref{eq:CO-interface} and
\eqref{eq:CP-interface} prove
\begin{equation}\label{eq:boundary-zeta-lower}
 \frac{\eta_c}{2}r
 \leq\nustar\left(B_{\frac{r}{2}}(p)\cap\overline\Omega\right)
 \leq C_c\zetastar(B_r(p)\cap\overline\Omega).
\end{equation}
Finally,
\begin{equation}\label{eq:zeta-below-nu}
 0\leq\upzeta_j\leq\upnu_j
 \quad\Rightarrow\quad0\leq\zetastar\leq\nustar.
\end{equation}
Equations \eqref{eq:boundary-zeta-upper-physical}--
\eqref{eq:zeta-below-nu} prove \eqref{eq:two-sided-boundary-density}.
\end{proof}

\begin{lem}
\label{lem:boundary-AC-covering}
Let $\lambda$ be a non-negative Radon measure on $\overline\Omega$.  If, on every
compact boundary arc, $\lambda(B_r(p))\leq Cr$ for all sufficiently small
$r$, then
its boundary restriction is absolutely continuous and satisfies
\[
 \lambda\llcorner\partial\Omega
 \ll\cH^1\llcorner\partial\Omega.
\]
Its Radon--Nikodym density is locally bounded above.
\end{lem}

\begin{proof}
If $A$ is a subset of a compact arc with $\cH^1(A)=0$, cover it by
boundary-centred balls $B_{r_i}(p_i)$ with
$\sum_i r_i<\delta$.  Then
\[
 \lambda(A)\leq\sum_i\lambda(B_{r_i}(p_i))
 \leq C\sum_i r_i<C\delta.
\]
Let $\delta\downarrow0$.  The density bound follows from differentiation
on the $C^1$ curve.
\end{proof}

\section{Proof of Theorem~\ref{thm:main}}
\label{sec:proof-main}

This section assembles the preceding estimates in five steps and proves
Theorem~\ref{thm:main}.

\subsection{Support and convergence away from the interface}

We first identify the common support of the limiting measures and establish
uniform convergence away from it.

Proposition~\ref{prop:unconditional} provides the phase and measure
compactness.  Lemma~\ref{lem:global-neumann-Linf}, a finite conformal cover
of $\partial\Omega$, and Corollary~\ref{cor:var-measure-package} yield
\eqref{eq:CO-interface}--\eqref{eq:CP-interface}.  Define
\begin{equation}\label{eq:Sstar-definition}
 \Sstar:=\supp\nustar.
\end{equation}
If $x\notin\Sstar$, choose $r>0$ such that
$\nustar(\overline{B_{2r}(x)}\cap\overline\Omega)=0$.  At radii carrying no
limiting boundary mass, \eqref{eq:var-small-inner} and scaling yield
\begin{equation}\label{eq:off-interface-energy-small}
 \cE_{\eps_j}(u_j;B_r(x)\cap\Omega)
 \leq C\frac{\eps_j}{r}
 \cE_{\eps_j}(u_j;B_{2r}(x)\cap\Omega)
 =o(\eps_j).
\end{equation}
The confinement conclusion of Theorem~\ref{thm:var-clearing-out} and the
connectedness of the smaller ball select $\upsigma_j\in\Sigma$ such that
\begin{equation}\label{eq:off-interface-confinement}
 \abs{u_j-\upsigma_j}\leq\frac{\upmu_V}{2},
 \quad
 \int_{B_r(x)\cap\Omega}V(u_j)\,\dd x
 \leq\eps_j\cE_{\eps_j}(u_j;B_r(x)\cap\Omega)
 =o(\eps_j^2).
\end{equation}
The local quadratic bound for $V$ implies
\[
 \int_{B_r(x)\cap\Omega}\abs{u_j-\upsigma_j}^2\,\dd x\to0.
\]
Since $u_j\to u_0$ in $L^1(\Omega)$ and $\Sigma$ is finite, the sequence
$\upsigma_j$ is eventually constant: two distinct wells occurring along
infinite subsequences would both agree almost everywhere with $u_0$ on
$B_r(x)\cap\Omega$.  Denote the eventual value by
$\upsigma=\upsigma(x)$ and discard the finitely many preceding indices.
In an interior ball, $w_j:=\abs{u_j-\upsigma}^2$ satisfies
\begin{equation}\label{eq:off-interface-subharmonic}
 \Delta w_j
 =2\abs{\nabla u_j}^2
 +2\eps_j^{-2}\nabla V(u_j)\cdot(u_j-\upsigma)\geq0,
 \quad
 w_j\leq c_V^{-1}V(u_j).
\end{equation}
Hence
\begin{equation}\label{eq:off-interface-uniform-convergence}
 \sup_{B_{\frac{r}{2}}(x)}\abs{u_j-\upsigma}^2
 \leq\frac{C}{r^2}\int_{B_r(x)}w_j\,\dd x
 \leq\frac{C}{r^2}\int_{B_r(x)}V(u_j)\,\dd x
 \to0.
\end{equation}
For $x\in\partial\Omega$, the even reflection in a conformal chart obeys
the same inequality
\begin{equation}\label{eq:off-interface-boundary-subharmonic}
 \Delta\abs{v_j-\upsigma}^2
 =2\abs{\nabla v_j}^2
 +2a\eps_j^{-2}\nabla V(v_j)\cdot(v_j-\upsigma)\geq0,
\end{equation}
so \eqref{eq:off-interface-uniform-convergence} follows from the full-ball
mean-value inequality.  Finally,
\begin{equation}\label{eq:support-equality-proof}
 \zetastar\leq\nustar
 \Rightarrow\supp\zetastar\subset\Sstar,
 \quad
 \eqref{eq:CP-interface}
 \Rightarrow\Sstar\subset\supp\zetastar,
\end{equation}
and therefore $\Sstar=\supp\nustar=\supp\zetastar$.

\subsection{Length and rectifiability}

The clearing-out lower bound controls the length and completes the
rectifiability argument at the boundary.

For $\delta\in(0,r_c)$, the fine $5r$ covering theorem yields pairwise
disjoint balls $B_{r_i}(x_i)$, $x_i\in\Sstar$, $r_i<\delta$, such that
\begin{equation}\label{eq:five-r-cover}
 \Sstar\subset\bigcup_iB_{5r_i}(x_i),
 \quad
 \eta_cr_i\leq\nustar(B_{r_i}(x_i)\cap\overline\Omega).
\end{equation}
Thus
\begin{equation}\label{eq:Sstar-length-bound}
 \cH^1_{10\delta}(\Sstar)
 \leq10\sum_ir_i
 \leq\frac{10}{\eta_c}\sum_i\nustar(B_{r_i}(x_i))
 \leq\frac{10M_0}{\eta_c}.
\end{equation}
Letting $\delta\downarrow0$ yields the asserted $\cH^1$ bound.  If
$U_m\subset\subset\Omega$ exhausts $\Omega$, interior clearing-out and
Theorem~\ref{thm:bethuel-interior} yields
\begin{equation}\label{eq:interior-carrier-identification}
 \Sstar\cap U_m=S_{U_m}
 \quad\text{up to an }\cH^1\text{-null set}.
\end{equation}
Therefore $\Sstar\cap\Omega$ is countably rectifiable, while
$\Sstar\cap\partial\Omega\subset\partial\Omega$ is rectifiable because
$\partial\Omega$ is $C^3$.

\subsection{Absolute continuity and density bounds}

We next combine the interior theorem with the boundary density estimates to
obtain global absolute continuity and uniform density comparisons.

Interior absolute continuity follows from
Theorem~\ref{thm:bethuel-interior}.  Boundary absolute continuity follows
from Proposition~\ref{prop:boundary-linear-density} and
Lemma~\ref{lem:boundary-AC-covering}.  Hence
\begin{equation}\label{eq:global-AC-first}
 \nustar=e\,\cH^1\llcorner\Sstar,
 \quad
 \zetastar=\Uptheta\,\cH^1\llcorner\Sstar.
\end{equation}
Since $\abs{\upmu_{\star,ij}}\leq2\nustar$,
\begin{equation}\label{eq:tensor-AC}
 \upmu_{\star,ij}=m_{ij}\,\cH^1\llcorner\Sstar,
 \quad\abs{m_{ij}}\leq2e.
\end{equation}
By Proposition~\ref{prop:boundary-linear-density}, compactness of
$\partial\Omega$, and the total mass bound, after enlarging $C$ we have
\begin{equation}\label{eq:all-scale-boundary-upper}
 \zetastar(B_s(p)\cap\overline\Omega)\leq Cs
 \quad\text{for every }p\in\partial\Omega\text{ and }s>0.
\end{equation}
Indeed, the proposition applies at sufficiently small radii, while for all
larger radii the estimate follows from
$\zetastar(\overline\Omega)\leq M_0$.

To make the upper bound uniform up to the boundary, take
$x\in\Omega$, $d:=\dist(x,\partial\Omega)$, and
$p\in\partial\Omega$ with $\abs{x-p}=d$.  If
$0<r<\frac{d}{2}$, the interior
version of \eqref{eq:gronwall-monotonicity}, with $a\equiv1$, yields
\begin{equation}\label{eq:interior-centre-small-radius}
\begin{aligned}
 \frac{\zetastar(B_r(x))}{r}
 &\leq\frac{\zetastar\left(B_{\frac{d}{2}}(x)\right)}{\frac{d}{2}}
 \\
 &\leq\frac{2\zetastar\left(B_{\frac{3d}{2}}(p)\right)}{d}\leq C.
\end{aligned}
\end{equation}
If $r\geq\frac{d}{2}$, then
\begin{equation}\label{eq:interior-centre-large-radius}
 B_r(x)\subset B_{3r}(p),
 \quad
 \zetastar(B_r(x))\leq\zetastar(B_{3r}(p))\leq Cr.
\end{equation}
Thus the first bound in \eqref{eq:global-linear-upper} holds at every
radius.  Applying \eqref{eq:CP-interface} at radius $2r$ proves the second
bound for $0<r<\frac{r_c}{2}$; for $r\geq\frac{r_c}{2}$, it follows from
$\nustar(\overline\Omega)\leq M_0$.  We have therefore shown
\begin{equation}\label{eq:global-linear-upper}
 \zetastar(B_r(x))\leq Cr,
 \quad
 \nustar(B_r(x))\leq Cr
 \quad(x\in\overline\Omega,\ r>0).
\end{equation}
For $x\in\Sstar$, \eqref{eq:CO-interface} and
\eqref{eq:CP-interface} yield
\begin{equation}\label{eq:global-linear-lower}
 \eta_cr\leq\nustar(B_r(x)),
 \quad
 \frac{\eta_c}{2}r
 \leq C_c\zetastar(B_r(x))
 \quad(0<r<r_c).
\end{equation}
Differentiating \eqref{eq:global-linear-upper}--
\eqref{eq:global-linear-lower} with respect to
$\cH^1\llcorner\Sstar$ yields
\begin{equation}\label{eq:density-comparison-proof}
 0<c_0\leq\Uptheta\leq e\leq C_0\Uptheta\leq C_1
 \quad\text{for }\cH^1\text{-a.e. }x\in\Sstar.
\end{equation}

\subsection{Discrepancy and free-boundary stationarity}

The interior and boundary stress identities now determine the limiting
varifold and its first variation.

The interior identities of Theorem~\ref{thm:bethuel-interior} and the
boundary identities \eqref{eq:boundary-discrepancy-measure} yield
\begin{equation}\label{eq:global-discrepancy-proof}
 2\Uptheta=m_{nn}-m_{\tau\tau},
 \quad m_{n\tau}=0,
 \quad e=m_{nn}
 \quad\cH^1\text{-a.e. on }\Sstar.
\end{equation}
Hence, in the $(\tau,n)$ frame,
\begin{equation}\label{eq:stress-varifold-identification}
 \frac{\dd T}{\dd(\cH^1\llcorner\Sstar)}
 =eI-m
 =\begin{pmatrix}2\Uptheta&0\\0&0\end{pmatrix}_{(\tau,n)}
 =2\Uptheta\,\tau\otimes\tau,
 \quad
 T=2(\tau\otimes\tau)\zetastar.
\end{equation}
For every $X\in C^1(\overline\Omega;\R^2)$ with
$X\cdot\nOmega=0$ on $\partial\Omega$, equations
\eqref{eq:inner-variation} and \eqref{eq:stress-varifold-identification}
yield
\begin{equation}\label{eq:stationarity-proof}
\begin{aligned}
 0
 &=\int_{\overline\Omega}T:DX
 =2\int_{\Sstar}(\tau\otimes\tau):DX\,\dd\zetastar
 \\
 &=2\int_{\Sstar}\diver_{T_x\Sstar}X\,\dd\zetastar
 =2\delta\Vbold(\Sstar,\Uptheta)(X).
\end{aligned}
\end{equation}

\subsection{Branch weights and boundary balance}

Finally, stationarity determines the weights on isolated branches and the
balance law at each finite-type boundary junction.

Let $\gamma:(a,b)\to\Sstar$ be an isolated interior $C^1$ arc parametrized
by arclength, and write $\tau=\gamma'$.  In a local graph chart, the trace
$g(s)v$ of a scalar function $g\in C_c^1((a,b))$ and a constant vector
$v\in\R^2$ has a $C^1$ ambient extension supported near the arc.
Stationarity implies
\begin{equation}\label{eq:branch-constant-weight}
 0=\int_a^b\Uptheta(s)\,
 \tau(s)\cdot v\,g'(s)\,\dd s
 \quad\text{for every }v\in\R^2.
\end{equation}
Thus $\partial_s(\Uptheta\tau)=0$ in
$\mathscr D'((a,b);\R^2)$.  Since $\abs{\tau}=1$ and $\Uptheta>0$, both
$\Uptheta$ and $\tau$ are constant on the arc.  On an isolated
boundary-supported arc, use instead the $C^2$ boundary tangent field; the
identity $\tau\cdot\tau'=0$ yields
$\int_a^b\Uptheta g'\,\dd s=0$.  Therefore every punctured branch in
Definition~\ref{defn:finite-type-boundary} carries a positive constant
weight $\Uptheta_\ell$.

Fix a finite-type junction $p$, let $v\in T_p\partial\Omega$, and choose
$\widetilde v\in C^1(\overline\Omega;\R^2)$ such that
\begin{equation}\label{eq:tangent-extension}
 \widetilde v(p)=v,
 \quad\widetilde v\cdot\nOmega=0\text{ on }\partial\Omega.
\end{equation}
For $\chi\in C_c^\infty([0,+\infty))$, $\chi=1$ near zero, set
\begin{equation}\label{eq:junction-test-field}
 X_\rho(x):=\chi\left(\frac{\abs{x-p}}{\rho}\right)\widetilde v(x).
\end{equation}
The $C^1$ parametrizations satisfy
\begin{equation}\label{eq:branch-first-order-expansion}
 \gamma_\ell(s)=p+s\tau_\ell+o(s),
 \quad
 \gamma_\ell'(s)=\tau_\ell+o(1).
\end{equation}
After the substitution $s=\rho t$,
\begin{equation}\label{eq:single-branch-endpoint-limit}
\begin{aligned}
 \lim_{\rho\downarrow0}
 \int_{\gamma_\ell}\diver_{T\gamma_\ell}X_\rho\,
 \Uptheta_\ell\,\dd\cH^1
 &=\Uptheta_\ell(v\cdot\tau_\ell)
 \int_0^\infty\chi'(t)\,\dd t
 \\
 &=-\Uptheta_\ell v\cdot\tau_\ell,
\end{aligned}
\end{equation}
while
\begin{equation}\label{eq:junction-smooth-error}
 \left|\int_{\Sstar\cap B_\rho(p)}
 (\tau\otimes\tau):D\widetilde v\,
 \chi\left(\frac{\abs{x-p}}{\rho}\right)\dd\zetastar\right|
 \leq C\rho\norm{D\widetilde v}_\infty\to0.
\end{equation}
Summing \eqref{eq:single-branch-endpoint-limit} and using
\eqref{eq:stationarity-proof} yields
\begin{equation}\label{eq:junction-balance-proof}
 v\cdot\sum_\ell\Uptheta_\ell\tau_\ell=0
 \quad\forall v\in T_p\partial\Omega
 \quad\Leftrightarrow\quad
 P_{T_p\partial\Omega}\left(\sum_\ell\Uptheta_\ell\tau_\ell\right)=0.
\end{equation}
If exactly one non-boundary branch and no other branch are present, then
$P_{T_p\partial\Omega}\tau_1=0$, hence
$\tau_1\perp T_p\partial\Omega$.\qed

\begin{rem}
The theorem does not claim $\zetastar(\partial\Omega)=0$.  The
boundary-clustered scalar examples described in \eqref{eq:scalar-embedding}
belong to the vectorial class and show that such a claim, as well as
unconditional pointwise orthogonality, would be false.
\end{rem}

\section*{Acknowledgments}

This work is partially supported by the National Key R\&D Program of China
under Grant 2023YFA1008801.

The authors acknowledge the use of AI tools. All mathematical arguments and
proofs in the final manuscript were checked and written by the authors.
 
\bibliographystyle{plain}
\bibliography{references}

@article{AllenCahn1979,
  author  = {Allen, Samuel M. and Cahn, John W.},
  title   = {A microscopic theory for antiphase boundary motion and its application to antiphase domain coarsening},
  journal = {Acta Metallurgica},
  volume  = {27},
  number  = {6},
  year    = {1979},
  pages   = {1085--1095},
  doi     = {10.1016/0001-6160(79)90196-2}
}

@article{Baldo1990,
  author  = {Baldo, Sisto},
  title   = {Minimal interface criterion for phase transitions in mixtures of {Cahn--Hilliard} fluids},
  journal = {Annales de l'Institut Henri Poincar\'e C, Analyse non lin\'eaire},
  volume  = {7},
  number  = {2},
  year    = {1990},
  pages   = {67--90},
  doi     = {10.1016/S0294-1449(16)30304-3}
}

@article{Bethuel2025,
  author  = {Bethuel, Fabrice},
  title   = {Asymptotics for $2$-dimensional vectorial {Allen--Cahn} systems},
  journal = {Acta Mathematica},
  volume  = {234},
  number  = {2},
  year    = {2025},
  pages   = {189--314},
  doi     = {10.4310/ACTA.2025.v234.n2.a1}
}

@article{CahnHilliard1958,
  author  = {Cahn, John W. and Hilliard, John E.},
  title   = {Free energy of a nonuniform system. {I}. {Interfacial} free energy},
  journal = {The Journal of Chemical Physics},
  volume  = {28},
  number  = {2},
  year    = {1958},
  pages   = {258--267},
  doi     = {10.1063/1.1744102}
}

@article{DeMasi2021,
  author  = {De Masi, Luigi},
  title   = {Rectifiability of the free boundary for varifolds},
  journal = {Indiana University Mathematics Journal},
  volume  = {70},
  number  = {6},
  year    = {2021},
  pages   = {2603--2651},
  doi     = {10.1512/iumj.2021.70.9401}
}

@article{FonsecaTartar1989,
  author  = {Fonseca, Irene and Tartar, Luc},
  title   = {The gradient theory of phase transitions for systems with two potential wells},
  journal = {Proceedings of the Royal Society of Edinburgh. Section A. Mathematics},
  volume  = {111},
  number  = {1--2},
  year    = {1989},
  pages   = {89--102},
  doi     = {10.1017/S030821050002504X}
}

@article{GruterJost1986,
  author  = {Gr\"uter, Michael and Jost, J\"urgen},
  title   = {{Allard} type regularity results for varifolds with free boundaries},
  journal = {Annali della Scuola Normale Superiore di Pisa. Classe di Scienze. Serie IV},
  volume  = {13},
  number  = {1},
  year    = {1986},
  pages   = {129--169}
}

@article{HutchinsonTonegawa2000,
  author  = {Hutchinson, John E. and Tonegawa, Yoshihiro},
  title   = {Convergence of phase interfaces in the van der {Waals--Cahn--Hilliard} theory},
  journal = {Calculus of Variations and Partial Differential Equations},
  volume  = {10},
  number  = {1},
  year    = {2000},
  pages   = {49--84},
  doi     = {10.1007/PL00013453}
}

@article{Ilmanen1993,
  author  = {Ilmanen, Tom},
  title   = {Convergence of the {Allen--Cahn} equation to {Brakke}'s motion by mean curvature},
  journal = {Journal of Differential Geometry},
  volume  = {38},
  number  = {2},
  year    = {1993},
  pages   = {417--461},
  doi     = {10.4310/JDG/1214454300}
}

@article{Kagaya2019,
  author  = {Kagaya, Takashi},
  title   = {Convergence of the {Allen--Cahn} equation with a zero {Neumann} boundary condition on non-convex domains},
  journal = {Mathematische Annalen},
  volume  = {373},
  number  = {3--4},
  year    = {2019},
  pages   = {1485--1528},
  doi     = {10.1007/s00208-018-1720-x}
}

@article{LiPariseSarnataro2024,
  author  = {Li, Martin Man-Chun and Parise, Davide and Sarnataro, Lorenzo},
  title   = {Boundary behavior of limit-interfaces for the {Allen--Cahn} equation on {Riemannian} manifolds with {Neumann} boundary condition},
  journal = {Archive for Rational Mechanics and Analysis},
  volume  = {248},
  year    = {2024},
  pages   = {Paper No. 124},
  doi     = {10.1007/s00205-024-02070-z}
}

@article{MalchiodiNiWei2007,
  author  = {Malchiodi, Andrea and Ni, Wei-Ming and Wei, Juncheng},
  title   = {Boundary-clustered interfaces for the {Allen--Cahn} equation},
  journal = {Pacific Journal of Mathematics},
  volume  = {229},
  number  = {2},
  year    = {2007},
  pages   = {447--468},
  doi     = {10.2140/pjm.2007.229.447}
}

@article{MalchiodiWei2007,
  author  = {Malchiodi, Andrea and Wei, Juncheng},
  title   = {Boundary interface for the {Allen--Cahn} equation},
  journal = {Journal of Fixed Point Theory and Applications},
  volume  = {1},
  number  = {2},
  year    = {2007},
  pages   = {305--336},
  doi     = {10.1007/s11784-007-0016-7}
}

@article{MizunoTonegawa2015,
  author  = {Mizuno, Masashi and Tonegawa, Yoshihiro},
  title   = {Convergence of the {Allen--Cahn} equation with {Neumann} boundary conditions},
  journal = {SIAM Journal on Mathematical Analysis},
  volume  = {47},
  number  = {3},
  year    = {2015},
  pages   = {1906--1932},
  doi     = {10.1137/140987808}
}

@article{Moser2023,
  author  = {Moser, Maximilian},
  title   = {Convergence of the scalar- and vector-valued {Allen--Cahn} equation to mean curvature flow with {$90^\circ$}-contact angle in higher dimensions, part {I}: Convergence result},
  journal = {Asymptotic Analysis},
  volume  = {131},
  number  = {3--4},
  year    = {2023},
  pages   = {297--383},
  doi     = {10.3233/ASY-221775}
}

@article{Modica1985,
  author  = {Modica, Luciano},
  title   = {A gradient bound and a {Liouville} theorem for nonlinear {Poisson} equations},
  journal = {Communications on Pure and Applied Mathematics},
  volume  = {38},
  number  = {5},
  year    = {1985},
  pages   = {679--684},
  doi     = {10.1002/cpa.3160380515}
}

@article{Modica1987,
  author  = {Modica, Luciano},
  title   = {The gradient theory of phase transitions and the minimal interface criterion},
  journal = {Archive for Rational Mechanics and Analysis},
  volume  = {98},
  number  = {2},
  year    = {1987},
  pages   = {123--142},
  doi     = {10.1007/BF00251230}
}

@article{ModicaMortola1977,
  author  = {Modica, Luciano and Mortola, Stefano},
  title   = {Un esempio di {$\Gamma$}-convergenza},
  journal = {Bollettino della Unione Matematica Italiana B},
  series  = {5},
  volume  = {14},
  number  = {1},
  year    = {1977},
  pages   = {285--299}
}

@book{Pommerenke1992,
  author    = {Pommerenke, Christian},
  title     = {Boundary Behaviour of Conformal Maps},
  series    = {Grundlehren der mathematischen Wissenschaften},
  volume    = {299},
  publisher = {Springer-Verlag},
  address   = {Berlin},
  year      = {1992},
  isbn      = {978-3-540-54751-8},
  doi       = {10.1007/978-3-662-02770-7}
}

@article{Warschawski1935,
  author  = {Warschawski, Stefan E.},
  title   = {On the higher derivatives at the boundary in conformal mapping},
  journal = {Transactions of the American Mathematical Society},
  volume  = {38},
  number  = {2},
  year    = {1935},
  pages   = {310--340},
  doi     = {10.1090/S0002-9947-1935-1501813-X}
}

@book{Simon1983,
  author    = {Simon, Leon},
  title     = {Lectures on Geometric Measure Theory},
  series    = {Proceedings of the Centre for Mathematical Analysis, Australian National University},
  volume    = {3},
  publisher = {Australian National University},
  address   = {Canberra},
  year      = {1983},
  isbn      = {0-86784-429-9}
}

@article{Tonegawa2003,
  author  = {Tonegawa, Yoshihiro},
  title   = {Domain dependent monotonicity formula for a singular perturbation problem},
  journal = {Indiana University Mathematics Journal},
  volume  = {52},
  number  = {1},
  year    = {2003},
  pages   = {69--84},
  doi     = {10.1512/iumj.2003.52.2351}
}

\end{document}